\documentclass[leqno,12pt]{amsart}
\usepackage{amsfonts}
\usepackage{amsmath,amssymb,amsthm}
\usepackage{amsfonts}

\usepackage{xcolor}

\usepackage{cleveref}

\newtheorem{theorem}{Theorem}[section]
\newtheorem{proposition}[theorem]{Proposition}
\newtheorem{lemma}[theorem]{Lemma}
\newtheorem{corry}[theorem]{Corollary}

\newtheorem{defi}[theorem]{Definition}
\theoremstyle{definition}
\newtheorem{example}[theorem]{Example}
\theoremstyle{definition}
\newtheorem{rem}[theorem]{Remark}

\numberwithin{claim}{theorem}

\renewenvironment{proof}{\textit{Proof.}}{\hfill\ensuremath{\qed}}

\def \qed{\hfill{\hbox{$\square$}}}

\renewenvironment{proof}{\textit{Proof.}}{\hfill\ensuremath{\qed}}

\numberwithin{equation}{section}

\begin{document}

\title{On  Time-Like Class  A Surfaces of a Warped Space-time with Lorentzian Fiber}

\author[F. Kaya]{Furkan Kaya}
\address{Department of Mathematics, Faculty of Science and Letters, Istanbul Technical University, Istanbul, T{\"u}rkiye}
\email{kayaf18@itu.edu.tr}

\author[N. Cenk Turgay]{Nurettin Cenk Turgay}
\address{Department of Mathematics, Faculty of Science and Letters, Istanbul Technical University, Istanbul, T{\"u}rkiye}
\email{turgayn@itu.edu.tr}


\begin{abstract}
In this paper, we consider time-like surfaces in the static space-time given by the warped product $\mathbb L^3_1(c)\, _f\times (I,dz^2)$, where $\mathbb L^3_1(c)$ denotes the Lorentzian space form with the constant sectional curvature $c\in\{-1,0,1\}$. In particular, we study the surfaces with light-like  $\left(\frac{\partial}{\partial z}\right)^T$. First, we construct a globally defined pseudo-orthonormal frame field on a surface satisfying this condition and deal with the invariants associated with this frame field. Then, we obtain a complete classification theorem for class~$\mathcal A$ surfaces. Finally, we consider some applications of this theorem.



\end{abstract}

\maketitle

\section{Introduction}
Let $M$ be a (semi-) Riemannian submanifold of $\bar{M}\times I$, where $I$ is an open interval and $M^m$ is a (se\-mi-) Riemannian manifold and $R^{n}(c)$ denote the $n$-dimensional Riemannian space form of constant sectional curvature $c$ with the metric tensor $g^{c}$. By considering the decomposition
\begin{equation}\label{Decompofpartialw}
\frac{\partial}{\partial w} = T + \eta,
\end{equation}
one can define a vector field $T$ tangent to and a vector field $\eta$ normal to $M$, where  $\frac{\partial}{\partial w}$ denotes the unit vector field tangent to the interval $I$.

Several recent works investigate submanifolds of warped product spaces, which play a significant role in modeling various physical space-times, \cite {Aliasetal2025arxiv,DemirciTurgaySen2025,DekimpeJVanderVeken2020,TurgaySen2025}. For example, Dekimpe and Van der Veken \cite{DekimpeJVanderVeken2020} provide an overview of the geometric properties of marginally trapped surfaces in Robertson--Walker space-times. Recently, Alias \textit{et al.} investigate codimension-two space-like submanifolds within the null hypersurfaces of generalized Robertson--Walker space-times in \cite{Aliasetal2025arxiv} and some geometrical properties of class~$\mathcal A$  surfaces of Robertson--Walker space-times are obtained in \cite{DemirciTurgaySen2025, TurgaySen2025}.

Based on the decomposition \eqref{Decompofpartialw}, the natural question of ``\textit{how the geometric constraints imposed on $T$ and $\eta$ determine the submanifold itself}'' arises. The first systematic studies in this direction were carried out in \cite{MendonTojeiro2014,Tojeiro2010}, where, in particular, the notion of class~$\mathcal A$ immersions was introduced as follows:
\begin{defi}\cite{Tojeiro2010}
Let $\phi: M^m \to  R^{n}(c)\times I$ be an isometric immersion from a Riemannian manifold $M$. The immersion $\phi$ is said to be of \emph{class~$\mathcal A$ } if the tangent component $T$ is an eigenvector of the shape operator at each point of $M$.
\end{defi}

It turns out that class~$\mathcal A$ submanifolds have some connections with several well-studied submanifolds of product spaces. For example, a biconservative submanifold is necessarily class~$\mathcal A$ under some conditions as obtained by Manfio et al. in \cite{Manfioetall2019}, where the authors study biconservative submanifolds of $R^{4}(c) \times \mathbb{R}$. Furthermore, the link between class~$\mathcal A$ submanifolds and those with constant principal curvatures is explored in a recent study in \cite{Manfioetall2025}, which focuses on hypersurfaces in $R^{3}(c) \times \mathbb{R}$. 

On the other hand, a classical and fundamental example of a warped product space-time is the Robertson--Walker model $I^1_1 \times_f R^{3}(c)$, which is defined on the product manifold $I \times R^{3}(c)$ with the metric tensor
\[
-dt^{2} + f(t)^{2} g^{c},
\]
where $f: I \to (0,\infty)$ is a smooth, non-vanishing, warping function depending only on the time-like coordinate $t$, \cite{Robertson1935,Walker1937}. In this construction, the vector field tangent to the base $I$ is time-like, and thus the Robertson--Walker space-time describes a time-oriented warped product in which the warping function evolves in the temporal direction. In contrast,  Dobarro and Ünal studied the geometric properties of standard static space-times, focusing on the characterization of their warped product structure, geodesic completeness, and curvature properties in \cite{DobarroUnal2004}. In this paper, based on their work, we study submanifolds of a space-time defined on the product manifold $\mathbb L^{3}_{1}(c) \times I$, equipped with the warped product metric
\begin{equation}\label{L31cfxIMetricTensor}
\widetilde  g=f(z)^{2} g_{c} + dz^{2},
\end{equation}
where  $\mathbb L^{n}_{1}(c)$ denotes the $n$-dimensional Lorentzian space form with sectional curvature $c$ equipped with the metric tensor $g_{c}$ and the warping function $f: I \to (0,\infty)$ depends on a spatial coordinate $z$.

Unlike Robertson--Walker spacetimes, where the warping function depends on the time coordinate, the warped space-time considered in this paper is obtained by warping a three-dimensional Lorentzian space form along the spatial direction $z$. Consequently, the resulting geometry is generally not an Einstein manifold unless the warping function satisfies a particular differential equation (cf.~\cite{ONeill1982,DobarroUnal2005}). Nevertheless, according to the standard interpretation of Einstein's field equations, the Einstein tensor may be regarded as an effective stress--energy tensor. This provides a natural physical interpretation of the underlying geometry as an effective anisotropic matter distribution, a viewpoint frequently adopted in the study of warped space-times and relativistic matter models (see, e.g., \cite{Visser1995,HerreraSantos1997}):

\begin{rem} For an arbitrary smooth function $f$ with the condition
$$ff''\neq f'^2-c,$$
 the metric $\widetilde  g$ does not satisfy the vacuum Einstein equations. However, by interpreting the Einstein tensor as an effective stress-energy tensor, the corresponding matter source may be regarded as an anisotropic fluid. In this case, the geometry of the space-time determine the stress energy tensor $T$ of Hawking–Ellis Type I with the form
$$T^\mu_{\ \nu}=\operatorname{diag}(-\rho,p_\perp,p_\perp,p_\parallel),$$
where the principal pressures satisfy
$$p_x=p_y=p_\perp\neq p_\|=p_z. $$
\end{rem}

On the other hand, the structural difference between the two models has important geometric implications.  In particular, from the viewpoint of the theory of \textit{non-degenerated} submanifolds, the space-like base in the static space-time  $\mathbb L^3_1(c)\,{}_f\times I$ allows for different geometric behaviors compared to the time-like base of Robertson--Walker space-times. One of the main differences is the causality of the vector fields $T$ and $\eta$ defined by the decomposition \eqref{Decompofpartialw}. Unlike the Robertson--Walker space-times, where $T$ and $\eta$ have opposite causal character---that is, one of them is space-like while the other is time-like---in the space-time $\mathbb{L}^3_1(c)\,{}_f\times I$, either $T$ or $\eta$ can be light-like.

In this direction, we investigate surfaces in the static space-time $\mathbb L^3_1(c)\,{}_f\times I$ with light-like $\left(\frac{\partial}{\partial z}\right)^T$. Section~2 recalls the fundamental notions of submanifold theory and fixes the notation that will be used throughout the paper. In Section~3, after constructing explicit examples of surfaces with light-like $\left(\frac{\partial}{\partial z}\right)^T$, we derive a characterization of surfaces with this property. Section~4 contains the main result, \Cref{L31cfI-ClassA-MainThm}, which provides a complete local classification of class~$\mathcal A$ surfaces in $\mathbb L^3_1(c)\,{}_f\times I$. Finally, in Section~5, we present applications of \Cref{L31cfI-ClassA-MainThm}, yielding all pseudo-umbilical and totally umbilical surfaces with light-like $\tfrac{\partial}{\partial z}$.


\section{Basic Concepts and Notation}
This section presents the basic notation used throughout the paper, along with a short summary of the theory of submanifolds of semi-Riemannian manifolds.

Let $\mathbb E^{n}_r$ denote the $n$-dimensional semi-Euclidean space with the index $r$ given by the metric tensor
$$g_{0,r}=-\sum\limits_{i=1}^rdx_i^2+\sum\limits_{i=r+1}^ndx_i^2,$$
where $(x_1,x_2,\hdots,x_n)$ is a rectangular coordinate system in $\mathbb R^n$ and we put $(x_1,x_2,x_3,x_4)=(t,x,y,z)$ if $n=4$.   

When $n>2$, $\mathbb L^n_1(c)$ stands for the $n$ dimensional Lorentzian space-form with the  constant sectional curvature $c$,  i.e.,
$$\mathbb L^n_1(c)=\left\{\begin{array}{cc}
\mathbb S^{n}_1&\mbox{if }c=1,\\
\mathbb E^{n}_1&\mbox{if }c=0,\\
\mathbb H^{n}_1&\mbox{if }c=-1
\end{array}\right.,\quad$$
and $g_c$ stands for its metric tensor, where $\mathbb S^n_1(K_0)$ and  $\mathbb H^n_1(K_0)$ will stand for $n$-dimensional de Sitter space and anti de-Sitter spaces, respectively, with the  sectional curvature $K_0$ and we put $\mathbb S^n_1(1)=\mathbb S^n_1$ and  $\mathbb H^n_1(-1)=\mathbb H^n_1$.

\subsection{The Static Space-Time $\mathbb L^3_1(c)\,{}_f\times I$}
In this subsection, we are going to consider the static space-time given by the warped product $\mathbb L^3_1(c)\,{}_f\times I$ with the Levi-Civita connection
$\widetilde\nabla$ and metric tensor $\widetilde  g=\langle\cdot,\cdot\rangle$ defined by \eqref{L31cfxIMetricTensor} for a non-vanishing smooth function $f$.

Throughout this paper, $F$ will denote the function defined by
\begin{equation*}
F(z):=\int_{z_0}^z \frac{1}{f(\xi)}d\xi
\end{equation*}
for a $z_0\in I$. Moreover, for a given vector field $X$ 
tangent to $\mathbb L^3_1(c)\,{}_f\times I$, we define a vector field $\bar X$ tangent to $\mathbb L^3_1(c)$ and a function $X_4$ by
\begin{equation}\label{L31cxfDefX4andbarX}
 \bar X:=\Pi_*(X),\qquad X_4:=\langle X,\frac{\partial}{\partial z} \rangle,
\end{equation} 
where $\Pi:\mathbb L^3_1(c)\times I\to\mathbb L^3_1(c)$ is the canonical projection. Then, we write
$$X = (\bar{X}, X_4).$$
Occasionally, by a slight abuse of notation, we adopt the convention 
$$(\bar{X}, 0) = \bar{X}.$$

On the other hand, from \cite{ONeill1982} it is obtained that Levi-Civita connection $\widetilde\nabla$ of  $\mathbb L^3_1(c)\,{}_f\times I$ has the form
\begin{equation}\label{L31cxfILCconn}
\widetilde{\nabla}_X Y=\nabla^0_XY+\frac{f'}{f} \left(X_{4}\bar{Y}+Y_{4}\bar{X},-\langle\bar X,\bar Y\rangle\right),
\end{equation} 
where $\nabla^{0}$ denotes the Levi-Civita connection of the Cartesian product $\mathbb L^3_1(c)\times I$.

Now, let $M$ be a time-like surface in $\mathbb L^3_1(c)\,{}_f\times I$. Then, by letting $w=z$ in \eqref{Decompofpartialw}, a vector field $T$ tangent to $M$ and a vector field $\eta$ normal to $M$ are defined by 
\begin{equation}\label{partialz}
		\left.\frac{\partial}{\partial z}\right|_M=T+\eta.
\end{equation}
\begin{rem}\textbf{Assumptions.} 
We are going to exclude the trivial cases when $M$ contains an open part of a horizontal slice $\hat M\,{}_{f(z_0)}\times\{z_0\}$ or a vertical cylinder $M=\alpha\,{}_f\times I$ and also the case when it is contained in a totally geodesic hypersurface of $\mathbb L^3_1(c)\,{}_f\times I$, where $\hat M$ is a surface and $\alpha$ is a curve in the space form $\mathbb L^3_1(c)$. Therefore, throughout the paper, we shall assume that
\begin{itemize}
\item $T_p \neq 0$ and $\eta_p \neq 0$ at every point $p\in M$,
\item $M$ does not contain any open subset lying in a totally geodesic hypersurface of $\mathbb L^3_1(c)\,{}_f\times I$.
\end{itemize}
\end{rem}

\subsection{Basic Definitions and Facts in the Theory of Submanifolds}
Let $M^n$ be a semi-Riemannian submanifold of $(N,\hat g)$ with the Levi-Civita connection $\nabla$, the second fundamental form $h$, shape operator $A$ and normal connection $\nabla^\bot$. Then, for all vector fields $X$ and $Y$ tangent to $M$ and $\xi$ normal to $M$,  the Gauss and Weingarten formul\ae 
\begin{align*}
\begin{split}
\nonumber{\nabla}^N_X Y&=\nabla_X Y+h(X,Y),\\
\nonumber{\nabla}^N_X \zeta&=-A_\zeta X+\nabla^\perp_X \xi
\end{split}
\end{align*}
are satisfied, where $\nabla^N$ denotes the Levi-Civita connection of $N$. Note that $h$ and $A$ are related by
\begin{equation}\label{hArelby}
\hat g( h(X,Y),\xi)=\hat g( A_\xi X,Y).
\end{equation}
The mean curvature vector field $H$ of $M$ is defined by
$$H=\frac 1n \mathrm{tr\,} h.$$
$M$ is said to be umbilical along $\xi$ if there exists a function $a_\xi\in C^\infty (M)$ such that 
$$\hat g(A_\xi X,Y)=a_\xi\hat g(X,Y).$$
If $M$ is  umbilical along $H$, then it is called a pseudo-umbilical submanifold of $N$. Furthermore, if $M$ is umbilical along $\xi$ for every $\xi$ normal to $M$, then $M$ is said to be totally umbilical.

Let $\tilde R$ and $R$ stand for the curvature tensors of $M$ and $N$, respectively and $R^\perp$ denote the normal curvature tensor of $M$. Then the Codazzi, Gauss and Ricci equations take the form 
\begin{eqnarray}
\label{Codazzi} (\tilde R(X,Y)Z)^\perp&=& (\overline{\nabla}{_X} h)(Y,Z)-(\overline{\nabla}_Y h)(X,Z),\\
\label{GaussEq} R(X,Y)Z&=&A_{h(Y,Z)}X-A_{h(X,Z)}Y
\end{eqnarray}
and
\begin{equation}\label{Ricci}
(\tilde R(X,Y)\xi)^\perp=R^\perp (X,Y)\xi+h(A_\xi X,Y)-h(X,A_\xi Y),
\end{equation}
respectively, where the covariant derivative $(\overline{\nabla}_X h)(Y,Z)$ of the second fundamental form $h$  is defined by
\begin{align}
\nonumber(\overline{\nabla}_X h)(Y,Z)=\nabla_X ^\bot h(Y,Z)-h(\nabla_X Y,Z)-h(Y,\nabla_X Z).
\end{align}
$M$ is said to have flat normal bundle if $R^\perp=0$.


Now, let $\hat{M}$ be a time-like surface in the Lorentzian space form $\mathbb L^3_1(c)$, with metric tensor $\hat g$ and shape operator $\hat{S}$. Further, assume that $\hat{M}$ is not totally geodesic. Then, the Gaussian curvature $K$ of $\hat M$ is defined by
$$K=\frac{\langle R(X,Y)Y, X \rangle}{\langle X,X\rangle \langle Y,Y\rangle - \langle X,Y\rangle^2},$$
where $\{X,Y\}$ is a frame field for the tangent bundle of $\hat M$. $\hat M$ is said to be flat if $K$ vanishes identically on $\hat M$. In this case, for all $p\in M$ there exists a local coordinate system $\left(\mathcal N_p, \left(U,V\right)\right)$ such that $\mathcal N_p\ni p$ and
\begin{equation}\label{L31c-isothermal-null}
\hat g=-(dU dV+dV dU).
\end{equation}
We are going to use the following remark.
\begin{rem}\label{L31fxIFlatNBPropRem}
Let $\hat M$ be a time-like surface. Then, a direct computation yields that local coordinates $u, v$ satisfy  
\begin{equation}
\label{L31fxIFlatNBPropCaseEqIIMetric} \hat g=-\frac 1{f(u)^2}du^2+\frac {1}{f(u)}(du dv+dv du)
\end{equation}
and $\partial_v$ is proportional to $\partial_U$ if and only if $(u, v)$ and $\left(U,V\right)$ given in \eqref{L31c-isothermal-null} are related by
\begin{equation}
\label{L31fxIFlatNBPropRemEq1} U(u,v)=\frac{1}{2 c_1}F(u)-\frac{v}{c_1}+c_2,\qquad
V(u,v)=c_1F(u)
\end{equation}
for some constants $c_1\neq0,c_2$.
\end{rem}

It is well known that the matrix representation of $\hat{S}$, with respect to a suitable frame field, takes one of the forms
\begin{equation}\label{SubSectPrelL31c}
\mathrm{I.}\, \hat{S} = \begin{pmatrix} \lambda_1 & 0 \\ 0 & \lambda_2 \end{pmatrix}, \qquad
\mathrm{II.}\, \hat{S} = \begin{pmatrix} \lambda & \mu \\ -\mu & \lambda \end{pmatrix}, \qquad
\mathrm{III.}\, \hat{S} = \begin{pmatrix} \lambda & 1 \\ 0 & \lambda \end{pmatrix}
\end{equation}
for some smooth functions $\lambda$, $\lambda_i$, and $\mu \neq 0$. Note that in {case III}, the frame field is pseudo-orthonormal, whereas in the remaining cases, it is orthonormal.

We are going to use the following well-known result:
\begin{rem}\label{L31cRemarkTotallUmb} 
If $\hat{M}^2_1$ is a (totally) umbilical surface, then it is locally an open part of one of the following surfaces (see, for example, \cite{Martinez2023}):
\begin{enumerate}
\item [(1)] $c=1$ and $\hat M\subset \mathbb S^2_1(\frac 1{r^2})\subset \mathbb S^3_1$ is  parametrized by
\begin{equation}\label{S31TotUmbParam1}
\tilde\phi(s_1,s_2)=(r\sinh s_1 , r\cosh s_1 \cos s_2,r\cosh s_1 \sin s_2,\sqrt{1-r^2}), \qquad 0<r^2<1, 
\end{equation}
\item [(2)] $c=-1$ and $\hat M\subset \mathbb H^2_1(-\frac 1{r^2})\subset \mathbb H^3_1$ is parametrized by
\begin{equation}\label{H31TotUmbParam1}
\tilde\phi(s_1,s_2)=(r\cosh s_1 \cos s_2,r\cosh s_1 \sin s_2,r\sinh s_1 ,\sqrt{r^2-1}), \qquad r^2>1, 
\end{equation}
\item [(3)] $c=-1$ and $\hat M\subset \mathbb H^3_1$ is  the flat surface parametrized by
\begin{align}\label{H31TotUmbQuadraParam1}
\begin{split}
\tilde\phi(U,V)=&\left(\frac{U+V}{\sqrt{2}},a (2 U V-1)-\frac{1}{4 a},a (2 U V-1)+\frac{1}{4 a},\frac{U-V}{\sqrt{2}}\right), \qquad a\neq 0, 
\end{split}
\end{align}
\item [(4)] $c=0$ and $\hat M\subset \mathbb S^2_1(\frac 1{r^2})\subset \mathbb E^3_1$ is parametrized by
\begin{equation}\label{E31TotUmbParam1}
\tilde\phi(s_1,s_2)=(r\sinh s_1 , r\cosh s_1 \cos s_2,r\cosh s_1 \sin s_2).
\end{equation}
\end{enumerate}
\end{rem}



\subsection{Null Scrolls}
Let $\hat M$ be a Lorentzian surface in $\mathbb L^3_1(c)$. The shape operator $\hat S$ of $\hat M$ has the canonical form of Case~III in \eqref{SubSectPrelL31c} if and only if $M$ is a null scroll, given in the following example (see,\textit{ e.g.}, the proof of \cite[Theorem on p.~55]{JiHuaHou2007}).
\begin{example}\label{L31cNullScroll}\cite{JiHuaHou2007,Kimx2year2003}
Let $\alpha$ be a light-like curve in $\mathbb L^3_1(c)$ with a Cartan frame $\{A,B;C\}$ such that
\begin{equation}\label{L31cNullScrollABCalphasatisfies2}
\langle A,B\rangle=-1,\ \langle C,C\rangle=1,\quad \langle A,A\rangle=\langle B,B\rangle=\langle C,A\rangle=\langle C,B\rangle=0
\end{equation}
 along $\alpha$ satisfying
\begin{equation}\label{L31cNullScrollABCalphasatisfies}
\alpha'=A,\quad A'= aC,\quad B'= bC+c\alpha,\quad C'=bA+aB
\end{equation}
for some smooth functions $a$ and $b$. Then, the surface parametrized by
\begin{equation}\label{L31c-NullScrollParam}
\tilde\phi(U, V)=\alpha(U)+V B(U)
\end{equation}
is called a null scroll. Note that if the function $b$ is constant, then the surface parametrized by \eqref{L31c-NullScrollParam} is called $B$-scroll.
\end{example}

\begin{rem}\cite{JiHuaHou2007}\label{FlatMinimalB-scrolls}
Let $\hat M$ be a null scroll in $\mathbb L^3_1(c)$ given by the parametrization \eqref{L31c-NullScrollParam}. Then, the shape operator $\hat S$ of $\hat M$ along the unit normal vector field
$$\tilde N=-tbB-C,$$
is, \cite{JiHuaHou2007}, 
\begin{equation}\label{L31c-NullScrollParamShapOp}
\left(
\begin{array}{cc}
 b & a+t b' \\
 0 & b \\
\end{array}
\right)\qquad\mbox{with respect to $\{\partial_V,\partial_U\}$.}
\end{equation}
Therefore, a null scroll in $\mathbb L^3_1(0)=\mathbb E^3_1$ is flat (and equivalently minimal) if and only if it is a B-scroll with $b=0$. In this case, if $a\neq0$, then $\hat M$ is congruent to the surface given by
\begin{equation}\label{E31-FlatBScrollParam}
\tilde\phi(U, V)=\frac1{6 \sqrt{2}}\left(U^3+6 U+6 V,3 \sqrt{2} U^2,U^3-6 U+6 V\right).
\end{equation}
\end{rem}


We are going to use the following lemma.
\begin{lemma}\label{FlatNullScrollH31-Lemma}
A flat null scroll $\hat M$  in the anti de Sitter space $\mathbb H^3_1$ generated by $a=-\frac 1{k^2}$ and $b=1$ is congruent to the B-scroll parametrized by 
\begin{align}\label{FlatNullScrollH31-LemmaEq1}
\begin{split}
\tilde\phi(U,V)=&\left(\frac{\left(U-2 k^2 V\right) \cos \left(\frac{U}{k}\right)-2 k \sin \left(\frac{U}{k}\right)}{2 k},-\frac{\left(2 k^3+k\right) \cos \left(\frac{U}{k}\right)+\left(U-2 k^2 V\right) \sin \left(\frac{U}{k}\right)}{2 \sqrt{2} k^2},\right.\\
&\left.\frac{k \left(2 k^2-1\right) \cos \left(\frac{U}{k}\right)-\left(U-2 k^2 V\right) \sin \left(\frac{U}{k}\right)}{2 \sqrt{2} k^2},\frac{\left(U-2 k^2 V\right) \cos \left(\frac{U}{k}\right)}{2 k}\right).
\end{split}
\end{align}
Conversely, the surface given by \eqref{FlatNullScrollH31-LemmaEq1} in $\mathbb H^3_1$ is flat.
\end{lemma}

\begin{proof}
By solving \eqref{L31cNullScrollABCalphasatisfies} for $c=-1$, $a=-\frac 1{k^2}$ and $b=1$, we get
\begin{align}\label{FlatNullScrollH31-LemmaProofEq1}
\begin{split}
\alpha(U)=&\frac{ \left(\sin \frac{U}{k}-\frac{U \cos \frac{U}{k}}{2 k}\right)}{k}v_1+\frac{ \left(-\frac{U \sin \frac{U}{k}}{2 k}-\frac{1}{2} \cos \frac{U}{k}\right)}{k}v_2+k \sin \frac{U}{k}v_3-k \cos \frac{U}{k} v_4,\\
A(U)=& \left(\frac{U \sin \frac{U}{k}}{2 k^3}+\frac{\sin \frac{U}{k} \sin \frac{2U}{k}}{4 k^2}+\frac{\cos ^3\frac{U}{k}}{2 k^2}\right)v_1+ \cos \frac{U}{k}v_3+ \sin \frac{U}{k}v_4\\&
+ \left(-\frac{U \cos \frac{U}{k}}{2 k^3}-\frac{\sin \frac{U}{k} \cos ^2\frac{U}{k}}{2 k^2}+\frac{\sin \frac{2U}{k} \cos \frac{U}{k}}{4 k^2}\right)v_2,\\
B(U)=& \cos \frac{U}{k}v_1+\sin \frac{U}{k}v_2, \\
C(U)=&-\frac{U \cos \frac{U}{k}}{2 k^2} v_1+\frac{ \left(k \cos \frac{U}{k}-U \sin \frac{U}{k}\right)}{2 k^2}v_2+k  \sin \frac{U}{k}v_3-k  \cos \frac{U}{k}v_4
\end{split}
\end{align}
for some constant vectors $v_i\in\mathbb E^4_1$. By a direct computation using \eqref{L31cNullScrollABCalphasatisfies2} and \eqref{FlatNullScrollH31-LemmaProofEq1}, we obtain
$$\langle v_1, v_{3} \rangle = -1, \qquad \langle v_3, v_{3} \rangle = \frac{1}{k^2}, \qquad \langle v_2, v_{4} \rangle = 1,$$
and $\langle v_i, v_{j} \rangle = 0$ for all other pairs. Therefore, up to a suitable isometry of $\mathbb H^3_1$, one can choose
$$v_1=(-k,0,0,-k),\ v_2=\left(0,\frac{1}{\sqrt{2}},\frac{1}{\sqrt{2}},0\right),\ v_3=\left(0,\frac{1}{\sqrt{2}},-\frac{1}{\sqrt{2}},0\right),\ v_4=\left(0,0,0,\frac{1}{k}\right)$$
from which, together with \eqref{FlatNullScrollH31-LemmaProofEq1}, we obtain \eqref{FlatNullScrollH31-LemmaEq1}.

Converse of the lemma follows from a direct computation.
\end{proof}


\section{Surfaces with Light-Like $T$}
In this section, we obtain local classifications of  time-like surfaces in $\mathbb L^3_1(c)\,{}_f\times I$ with light-like  $\left(\frac{\partial}{\partial z}\right)^T$.

Let $M$ be an oriented time-like surface in $\mathbb L^3_1(c)\,{}_f\times I$ and assume that the tangent vector field $T$ defined by 
\eqref{partialz} is light-like.  In this case, the decomposition \eqref{partialz} turns into 
\begin{equation}\label{partialzDecompT-LL}
    \left.\frac{\partial}{\partial z}\right|_M = T + e_3,
\end{equation}
where $e_3$ is a unit normal vector field. We are going to consider the pseudo-orthonormal frame field $\{T,U\}$ of the tangent bundle of $M$ such that 
$$\langle T,T\rangle=\langle U,U\rangle=0,\quad \langle T,U\rangle=-1$$
and an orthonormal frame field $\{e_3,e_4\}$ of the normal bundle of $M$.

\subsection{Examples of Surfaces}
In this subsection, we construct some examples of surfaces which satisfies certain geometrical properties.

First, we obtain the next proposition to present a surface which satisfies the property of having light-like $\left(\frac{\partial}{\partial z}\right)^T$.
\begin{proposition}\label{SubsectTisLLProp1}
Let $\tilde{\phi}(u,v)$ be a parametrization of a time-like surface $\hat M$ in $\mathbb L^{3}_{1}(c)$ with the induced metric
\begin{equation}
\label{SubsectTisLLProp1Eq2} \hat g=-\frac 1{f(u)^2}du^2+\frac {E(u,v)}{f(u)^2}(du dv+dv du)
\end{equation}
for a non-vanishing function $E$ and consider the surface $M$ of $\mathbb L^3_1(c)\,{}_f\times I$ parametrized by 
\begin{equation}\label{SubsectTisLLProp1Eq1} 
\phi(u,v)=\left(\tilde{\phi}(u,v),u\right).
\end{equation}
Then, $T=\left(\frac{\partial}{\partial z}\right)^T$ is light-like on $M$.
\end{proposition}
\begin{proof}
By considering \eqref{SubsectTisLLProp1Eq2}, we see that $T=\frac 1E\partial_v$  is light-like and the vector field $e_3=\frac{1}{E}(-\tilde{\phi}_v,E)$ is normal to $M$. Furthermore, $T$ and $e_3$ satisfy \eqref{partialzDecompT-LL}. Therefore, $\left(\frac{\partial}{\partial z}\right)^T=T$ is light-like on $M$.
\end{proof}

\begin{defi}
Throughout this paper, a surface $M$ of $\mathbb L^3_1(c)\,{}_f\times I$ parametrized by \eqref{SubsectTisLLProp1Eq1} for a surface $\hat M$ in   $\mathbb L^{3}_{1}(c)$  is going to be called as the surface\textit{ generated by $\hat M$}.
\end{defi}

By making certain specific choices for the surface $\hat{M}$ in \Cref{SubsectTisLLProp1}, we obtain the following explicit examples:
\begin{example} \label{SubsectTisLLProp1S31}
Let $c=1$ and  consider the surface $M$ generated by $\hat M\subset \mathbb S^2_1( \frac 1{r^2}),\ r<1$. Then, by a direct computation considering \eqref{S31TotUmbParam1} and \eqref{SubsectTisLLProp1Eq2}, we obtain the parametrization 
\begin{align}\label{S31TotUmbParam2} 
\begin{split}
\phi(u,v)=&\left(\frac{r}{A(u,v)},r \left(\frac{\cos (a(u))}{A(u,v)}-\sin (a(u))\right),r \left(\frac{\sin (a(u))}{A(u,v)}+\cos (a(u))\right),\right.\\
&\left.
\sqrt{1-r^2},u\right)
\end{split}
\end{align}
of $M$ in $ \mathbb {S}^{3}_1\,{}_f\times I$, where $A$ and $a$ are smooth, non-vanishing functions satisfying
\begin{equation}\label{L31cTotUmbParamAsatisfies}
A_u=-\frac{A ^2}{2 r^2 f ^2 a' }-\frac{1}{2} a'  \left(A ^2+1\right).
\end{equation}
Note that we have $T=\frac1E \partial_v$ and the vector fields
\begin{align*}
\begin{split}
e_3=&\frac{1}{r f ^2 a' }\left( 1,\cos a ,\sin a ,0,r f ^2 a' \right),\\
e_4=&\frac{\sqrt{1-r^2}}{f  A } \left( -1,A\sin a   -\cos a ,-A \cos a-\sin a ,\frac{r A }{\sqrt{1-r^2}},0\right)\\
\end{split}
\end{align*}
form an orthonormal base of the normal $M$ such that $\eta=e_3$, where the function $E$ appearing in \Cref{SubsectTisLLProp1} is
\begin{equation}\label{L31cTotUmbShape0}
E=\frac{r^2 f^2 a' A_v}{A^2}.
\end{equation}
By a direct computation we obtain the shape operators of $M$ as 
\begin{equation} \label{L31cTotUmbShape}
A_{e_3} = \begin{pmatrix}
-\frac{f' }{f } & \frac{a' }{A }+\frac{f' }{f }+\frac{a'' }{a' } \\
0 & -\frac{f' }{f } 
				\end{pmatrix}, \qquad
A_{e_4} = \begin{pmatrix}
-h^4_{12} & -h^4_{22}\\
 0 & -h^4_{12} 
\end{pmatrix},
\end{equation} 
which yields that the surface parametrized by \eqref{S31TotUmbParam2} is a class~$\mathcal A$  surface with light-like $\left(\frac{\partial}{\partial z}\right)^T$, where we have
\begin{equation} \label{L31cTotUmbShape2}
h^4_{12}=h^4_{22}=-\frac{\sqrt{1-r^2}}{r f }.
\end{equation} 
\end{example}


\begin{example} \label{SubsectTisLLProp1H31}
Let $c=-1$ and  consider the surface $M$ generated by $\hat M\subset \mathbb H^2_1(- \frac 1{r^2}),\ r>1$. Then, by a direct computation considering \eqref{H31TotUmbParam1} and \eqref{SubsectTisLLProp1Eq2}, we obtain the parametrization 
\begin{align}\label{H31TotUmbParam2} 
\begin{split}
\phi(u,v)=&\left(r \left(\frac{\cos  a(u) }{A(u,v)}-\sin  a(u) \right),r \left(\frac{\sin  a(u) }{A(u,v)}+\cos  a(u) \right),\frac{r}{A(u,v)},\right.\\
&\left.\sqrt{r^2-1},u\right)
\end{split}
\end{align}
of $M$ in $ \mathbb {H}^{3}_1\,{}_f\times I$, where $A$ and $a$ are smooth, non-vanishing functions satisfying 
$$A_u=\frac{A ^2}{2 r^2 f ^2 a' }-\frac{1}{2} a'  \left(A ^2+1\right).$$
 Similar to \Cref{SubsectTisLLProp1E31},  we observe that the surface $M$ is a class~$\mathcal A$  surface with light-like $\left(\frac{\partial}{\partial z}\right)^T$ by obtaining \eqref{L31cTotUmbShape} for
\begin{align*}
\begin{split}
e_3=&\frac{1}{r f ^2 a' }\left(-\cos a ,-\sin a ,-1,0,r f ^2 a' \right),\\
e_4=&\frac{\sqrt{r^2-1}}{f  A }\left(\cos a -A\sin a   ,A\cos a  +\sin a ,1,\frac{r A }{\sqrt{r^2-1}},0\right)\\
E=&-\frac{r^2 f^2 a' A_v}{A^2},\qquad h^4_{12}=h^4_{22}=\frac{\sqrt{r^2-1}}{r f } 
\end{split}
\end{align*}
\end{example}

\begin{example} \label{SubsectTisLLProp1H31Quad}
Let $c=-1$ and   consider the surface $M$ generated by  flat totally umbilical surface $\hat M$ parametrized by \eqref{H31TotUmbQuadraParam1}. Then, by  considering \eqref{H31TotUmbQuadraParam1} and \eqref{SubsectTisLLProp1Eq2}, we obtain the parametrization 
\begin{align*}
\begin{split}
\phi(u,v)=&\left(\frac{2 c_1^2 F(u)+2 c_2 c_1+F(u)-2 v}{2 \sqrt{2} c_1},-\frac{a \left(c_1 F(u)+c_2\right) (2 v-F(u))}{c_1}-a-\frac{1}{4 a},\right.\\
&\left.-\frac{a \left(c_1 F(u)+c_2\right) (2 v-F(u))}{c_1}-a+\frac{1}{4 a},\frac{\left(2 c_1^2-1\right) F(u)+2 \left(c_1 c_2+v\right)}{2 \sqrt{2} c_1},u\right)
\end{split}
\end{align*}
of $M$ in $ \mathbb {H}^{3}_1\,{}_f\times I$, where $a,\ c_1,$ and $c_2$ are some constants with $a,b\neq 0$. By a direct computation, we obtain the shape operators of $M$ as
\begin{align}\label{H31TotUmbQuadShpOp} 
\begin{split}
A_{e_3}=\frac {f'}f I,\qquad A_{e_4}=
\begin{pmatrix}
 \frac{1}{f} & \frac{1}{f} \\
 0 & \frac{1}{f} \\
\end{pmatrix},
\end{split}
\end{align}
where we have
\begin{align*}
\begin{split}
e_3=&\frac{1}{\sqrt{2} c_1 f}\left(1,2 \sqrt{2} a \left(c_1 F+c_2\right),2 \sqrt{2} a \left(c_1 F+c_2\right),-1,\sqrt{2} c_1 f \right),\\
e_4=&\frac{1}{2 \sqrt{2} c_1 f}\left(-2 c_1^2 F-2 c_2 c_1-F+2 v,2 \sqrt{2} c_1 \left(\frac{a \left(c_1 F+c_2\right) (2 v-F)}{c_1}-a+\frac{1}{4 a}\right),\right.\\
&\left.-2 \sqrt{2} c_1 \left(-\frac{a \left(c_1 F+c_2\right) (2 v-F)}{c_1}+a+\frac{1}{4 a}\right),\left(1-2 c_1^2\right) F-2 \left(c_1 c_2+v\right),0\right)\\
E=&f.
\end{split}
\end{align*}
Hence, $M$ is a class~$\mathcal A$  surface with light-like $\left(\frac{\partial}{\partial z}\right)^T$.
\end{example}


\begin{example} \label{SubsectTisLLProp1E31}
Let $c=0$ and   consider the surface $M$ generated by $\hat M\subset \mathbb S^2_1(\frac 1{r^2})$. Then, by a direct computation considering \eqref{E31TotUmbParam1} and \eqref{SubsectTisLLProp1Eq2}, we obtain the parametrization 
\begin{align*} 
\begin{split}
\phi(u,v)=&\left(\frac{r}{A(u,v)},r \left(\frac{\cos a(u)}{A(u,v)}-\sin a(u)\right),r \left(\frac{\sin a(u)}{A(u,v)}+\cos a(u)\right),u\right)
\end{split}
\end{align*}
of $M$ in $ \mathbb {E}^{3}_1\,{}_f\times I$, where $A$ and $a$ are smooth, non-vanishing functions satisfying \eqref{L31cTotUmbParamAsatisfies}.
Similar to \Cref{SubsectTisLLProp1S31},  we observe that the surface $M$ is a class~$\mathcal A$  surface with light-like $\left(\frac{\partial}{\partial z}\right)^T$ by obtaining \eqref{L31cTotUmbShape} for
\begin{align*}
\begin{split}
e_3=&\frac{1}{r f^2 a'}\left(1,{\cos  a },{\sin  a },{r f^2 a'}\right),\\
e_4=&-\frac 1{Af}\left(1,\cos  a -A\sin  a,A\cos  a+\sin  a ,0\right)\\
E=&\frac{r^2 f^2 a' A_v}{A^2},\qquad h^4_{12}=h^4_{22}=-\frac{1}{r f } 
\end{split}
\end{align*}
\end{example}


\begin{example} \label{SubsectTisLLProp1L31c}
Consider the surface $M$ in  $\mathbb L^3_1(c)\,{}_f\times I$ generated by a null scroll  $\hat M$  in $ \mathbb L^3_1(c)$ described in \Cref{L31cNullScroll} for some smooth functions $a$ and $b$. Then, by a direct computation considering \eqref{L31c-NullScrollParam} and \eqref{SubsectTisLLProp1Eq2}, we obtain the parametrization 
\begin{align}\label{L31cTotUmbParam2} 
\begin{split}
\phi(u,v)=&\left(\alpha(U(u))+V(u,v) B(U(u)),u\right)
\end{split}
\end{align}
of  $M$ in $\mathbb L^{3}_1(c)\,{}_f\times I$, where $V$ and $U$ are smooth, non-vanishing functions satisfying 
\begin{equation}\label{L31cParamAsatisfies}
 U'  V ^2 \left(b(U )^2+c\right)-2 V_u  =-\frac1{f^2U'  }\, .
\end{equation}
Similar to \Cref{SubsectTisLLProp1S31},  we observe that the shape operators of $M$ has the form
\begin{equation} \label{L31cB-ScrollShape}
A_{e_3} = \begin{pmatrix}
 -\frac{f'}{f} & \frac{f' U'+f \left(\left(b(U)^2+c\right) V U'{}^2+U''\right)}{f U'} \\
 0 & -\frac{f'}{f} 
				\end{pmatrix}, \qquad
A_{e_4} = \begin{pmatrix}
 \frac{b(U)}{f} & \frac{f^2 \left(a+V b'\right) U'{}^2+b(U)}{f} \\
 0 & \frac{b(U)}{f} \\
\end{pmatrix},
\end{equation}  
which yields that $M$ is a class~$\mathcal A$  surface with light-like $\left(\frac{\partial}{\partial z}\right)^T$, where we have
\begin{eqnarray}
\nonumber e_3&=&\frac 1{U'f^2}\left( B(U) ,U'f^2\right),\\
\nonumber e_4&=&\frac {1}{f}\left( Vb(U)B(U)+C(U),0\right)\\
\label{L31cParamEDef}E&=&-U' f^2 V_v .
\end{eqnarray}
\end{example}

\subsection{Surfaces with Light-Like $T$}

In this subsection, we consider time-like surfaces under the condition that the vector field $T=\left(\frac{\partial}{\partial z}\right)^T$ is light-like.

Let $M$ be an oriented time-like surface in the space-time $\mathbb L^3_1(c)\,{}_f\times I$ such that the tangent vector field $T$, defined by \eqref{partialz}, is light-like at every point of $M$. In this case, using equation \eqref{partialzDecompT-LL}, one can construct a pseudo-orthonormal frame field $\{T, U\}$ for the tangent bundle of $M$ satisfying
$$\langle T,T\rangle=\langle U,U\rangle=0,\qquad \langle T,U\rangle=-1$$
as  well as an orthonormal frame field $\{e_3,e_4\}$ for the normal bundle of $M$. Note that \eqref{partialzDecompT-LL} implies
\begin{align}\label{SubsectTisLL-Eq1}
\begin{split}
T=(\bar T,0),\qquad& U=(\bar U,-1),\\
e_3=(\bar{e_3},1),\qquad& e_4=(\bar{e_4},0).
\end{split}
\end{align}

We are going to use the following lemma throughout this article:
\begin{lemma}
		Let $M$ be a time-like surface in $\mathbb L^3_1(c)\,{}_f\times I$ such that the vector field $\left(\frac{\partial}{\partial z}\right)^T$ is light-like on $M$. Then, the vector field $U$ defined by  \eqref{partialzDecompT-LL} satisfies
\begin{equation}\label{SubsectTisLLLemma0Eq1}
    \left.f'\right|_M=-U\left(\left.f\right|_M\right).
\end{equation}
\end{lemma}
\begin{proof}
Let $p=(\tilde p,z(p))\in M$ and consider an integral curve $\alpha=(\bar \alpha,\alpha_4)$ of $e_1$ starting from $p$. Then, we  have
$$U(\left.f\right|_M)_p=\left.\frac{d}{du}\right|_{u=0} (f\circ\alpha)(u)=\left.\frac{d}{du}\right|_{u=0} f(\alpha_4(u))$$
which yields
\begin{equation}\label{SubsectTisLLLemma0PEq1}
  U(\left.f\right|_M)_p=\alpha'_4(0)f'(\alpha_4(0)).
\end{equation}
Since $\alpha(0)=p$ and $\alpha'(0)=U_p$, \eqref{SubsectTisLL-Eq1} implies $\alpha_4'(0)=-1$. Hence, \eqref{SubsectTisLLLemma0PEq1} turns into
$$U(\left.f\right|_M)_p=\-f'(z(p)).$$
which yields \eqref{SubsectTisLLLemma0Eq1}.
\end{proof}


On the other hand, the Levi-Civita connection $\nabla$ of $M$ satisfies
\begin{align}\label{SubsectTisLLLCConn1ALL}
\begin{split}
\widetilde\nabla_{T}T=\omega_1T,\qquad \widetilde\nabla_{T}U=-\omega_1U,\\
\widetilde\nabla_{U}T=\omega_2T,\qquad \widetilde\nabla_{U}U=-\omega_2U,
\end{split}
\end{align}
while the normal connection $\nabla^\perp$ of $M$ induces 
\begin{align}\label{SubsectTisLLNrmCon1ALL}
\begin{split}
\widetilde\nabla^\perp_{T}e_3=\omega_3e_4,\qquad \widetilde\nabla^\perp_{T}e_4=-\omega_3e_3,\\
\widetilde\nabla^\perp_{U}e_3=\omega_4e_4,\qquad \widetilde\nabla^\perp_{U}e_4=-\omega_4e_3,
\end{split}
\end{align}
for some smooth functions $\omega_i$. Moreover, the second fundamental form $h$ of $M$ takes the form
\begin{align}\label{SubsectTisLL2ndFun1ALL}
\begin{split}
h(T,T) &=h^3_{11}e_{3}+ h^4_{11}e_{4}\\
h(T,U)&=h^3_{12}e_{3}+ h^4_{12}e_{4}\\
h(U,U)&=h^3_{22}e_{3}+ h^4_{22}e_{4},
\end{split}
\end{align}
where $h^a_{jk}$ are some smooth functions, with $a = 3,4$ and $j,k = 1,2$.


Now, we are ready to prove the following lemma:
\begin{lemma}\label{SubsectTisLLLemma1}
		Let $M$ be a time-like surface in $\mathbb L^3_1(c)\,{}_f\times I$ such that the vector field $\left(\frac{\partial}{\partial z}\right)^T$ is light-like on $M$   and consider the positively oriented, global frame frame field $\{T,U;e_3,e_4\}$ defined by \eqref{partialzDecompT-LL}. Then, the following conditions hold:
\begin{enumerate}
\item [(1)] The Levi-Civita connection $\nabla$ of $M$ satisfies
		\begin{align}\label{SubsectTisLLLCConn2}
			\nabla_T T = \nabla_T U = 0, \quad \nabla_U T = \left(\frac{f^{\prime}}{f}-h^{3}_{22}\right) T, \quad \nabla_U U = -\left(\frac{f^{\prime}}{f}-h^{3}_{22}\right)U.
		\end{align}
\item [(2)] The matrix representation of the shape operators of $M$ corresponding to the pseudo-orthonormal frame $\{T,U\}$ has the form
		\begin{equation} \label{SubsectTisLLShape1}
			\begin{aligned}
				A_{e_3} &= \begin{pmatrix}
					-\frac{f^{\prime}}{f} & -h^3_{22} \\
					0 & -\frac{f^{\prime}}{f}
				\end{pmatrix}, \qquad
				A_{e_4} &= \begin{pmatrix}
					-h^4_{12} & -h^4_{22} \\
					-h^4_{11} & -h^4_{12}
				\end{pmatrix}
			\end{aligned}
		\end{equation} 
\item [(3)]  The second fundamental form $h$ of $M$ satisfies
			\begin{align}\label{SubsectTisLL2ndFun2}
			h(T,T) = h^4_{11}e_{4} \quad h(T,U) = -\frac{f^{\prime}}{f}e_{3}+ h^4_{12}e_{4}, \quad h(U,U)= h^3_{22}e_{3}+h^4_{22}e_{4}.
		\end{align}
\item [(4)]  The  normal connection $\nabla^\perp$ of $M$ induces the following relations
		\begin{align}\label{SubsectTisLLNrmCon2}
			\nabla^{\perp}_{T}e_3=-h^4_{11}e_{4}, \quad \nabla^{\perp}_{U}e_3=-h^4_{12}e_{4}.
		\end{align} 
\end{enumerate}
\end{lemma}

	\begin{proof}
We are going to use the formul\ae\ given by \eqref{SubsectTisLL-Eq1} - \eqref{SubsectTisLL2ndFun1ALL}. Note that by combining  \eqref{SubsectTisLLNrmCon1ALL} with  \eqref{hArelby}, we have 	
\begin{align}\label{SubsectTisLLLemma1PrfEq0}   
\begin{split} 		
A_{e_3}T=-h^{3}_{12}T-h^{3}_{11}U,\qquad& A_{e_3}U=-h^{3}_{22}T-h^{3}_{12}U,\\
A_{e_4}T=-h^{4}_{12}T-h^{4}_{11}U,\qquad& A_{e_4}U=-h^{4}_{22}T-h^{4}_{12}U.
\end{split} 		
\end{align}
On the other hand, by getting the covariant derivative of \eqref{partialzDecompT-LL} along a vector field $X$ tangent to $M$, we obtain
\begin{equation}\label{SubsectTisLLLemma1PrfEq1}   
\widetilde\nabla_X \frac{\partial}{\partial z}= \nabla_X T+h(X,T)-A_{e_3} X+\nabla^\perp_X e_3.
\end{equation}
Moreover, \eqref{L31cxfILCconn} implies $\widetilde\nabla_X \frac{\partial}{\partial z}= \frac{f'}f\bar X$ which is equivalent to
\begin{equation}\label{SubsectTisLLLemma1PrfEq2}   
\widetilde\nabla_X \frac{\partial}{\partial z}= \frac{f'}f\left(X-\langle X,T\rangle\frac{\partial}{\partial z}\right)
\end{equation}
because of \eqref{L31cxfDefX4andbarX}. By considering   \eqref{partialzDecompT-LL} and \eqref{SubsectTisLLLemma1PrfEq2}, we observe that the tangential and normal parts of \eqref{SubsectTisLLLemma1PrfEq1} imply
\begin{equation}\label{SubsectTisLLLemma1PrfEq3}   
\frac{f'}f\left(X-\langle X,T\rangle T\right)= \nabla_X T-A_{e_3} X
\end{equation}
and
\begin{equation}\label{SubsectTisLLLemma1PrfEq4}   
-\frac{f'}f\langle X,T\rangle e_3= h(X,T)+\nabla^\perp_X e_3,
\end{equation}
respectively.

By combining \eqref{SubsectTisLL-Eq1}, \eqref{SubsectTisLLLCConn1ALL} and \eqref{SubsectTisLLLemma1PrfEq0} with \eqref{SubsectTisLLLemma1PrfEq3} for $X=T$ and $X=U$, we get
\begin{align}\label{SubsectTisLLLemma1PrfEq3-2} 
\begin{split} 		
\dfrac{f^{\prime}}{f}T&=h^{3}_{11}U+(\omega_1 +h^{3}_{12})T,\\
\dfrac{f^{\prime}}{f}\left(T+U\right)&=(\omega_2+h^{3}_{22})T+h^{3}_{12}U
\end{split} 		
\end{align}
and, because of 		\eqref{SubsectTisLLNrmCon1ALL}	and  \eqref{SubsectTisLL2ndFun1ALL}, \eqref{SubsectTisLLLemma1PrfEq4} for $X=T$ and $X=U$		imply
\begin{align}\label{SubsectTisLLLemma1PrfEq4-2} 
\begin{split} 		
0&=h^{3}_{11}e_{3}+(\omega_3+h^{4}_{11})e_{4},\\\
\dfrac{f^{\prime}}{f}\left(e_3\right)&=h^{3}_{12}e_3+(h^{4}_{12}+\omega_4)e_4.
\end{split} 		
\end{align} 		
 By a direct computation using \eqref{SubsectTisLLLemma1PrfEq3-2}, we obtain 
\begin{align}\label{SubsectTisLLLemma1PrfEq5} 
\begin{split} 	
			h^{3}_{11} = 0, &\quad	h^{3}_{12}=\dfrac{f^{\prime}}{f},\\
			\omega_1 + h^{3}_{12} = \dfrac{f'}{f},		&\quad	\omega_2+h^{3}_{22} = \dfrac{f'}{f},\\
\end{split} 		
\end{align} 
and the equations appearing in \eqref{SubsectTisLLLemma1PrfEq4-2} give
\begin{equation}\label{SubsectTisLLLemma1PrfEq6} 
			\omega_3=-h^{4}_{11}, \quad			\omega_4=-h^{4}_{12}.
\end{equation} 		
By using \eqref{SubsectTisLLLemma1PrfEq5}, we observe that  \eqref{SubsectTisLLLCConn1ALL}, \eqref{SubsectTisLLLemma1PrfEq0}  and  \eqref{SubsectTisLL2ndFun1ALL} turn into \eqref{SubsectTisLLLCConn2}, \eqref{SubsectTisLLShape1} and \eqref{SubsectTisLL2ndFun2}, respectively, and \eqref{SubsectTisLLLemma1PrfEq6}, together with \eqref{SubsectTisLLNrmCon1ALL}, yields \eqref{SubsectTisLLNrmCon2}. 
\end{proof}


Next, by using	\Cref{SubsectTisLLLemma1}, we construct a local coordinate system compatible with the global frame field $\{T,U;e_3,e_4\}$.
\begin{lemma}\label{SubsectTisLLLemma2}
		Let $M$ be a time-like surface in $\mathbb L^3_1(c)\,{}_f\times I$ such that the vector field $\left(\frac{\partial}{\partial z}\right)^T$ is light-like on $M$ and $p\in M$.  Then, there exists a local coordinate system  $\left(\mathcal N_p, \left(u,v\right)\right)$ such that $\mathcal N_p\ni p$ and the following conditions hold:
\begin{itemize} 
\item [(a)] The vector field $T$ and $U$ defined by \eqref{partialzDecompT-LL} satisfy
\begin{eqnarray}
\label{SubsectTisLLLemma2Eq1a} \left.T\right|_{\mathcal N_p}&=&\frac{1}{E}\partial_v,\\
\label{SubsectTisLLLemma2Eq1b} \left.U\right|_{\mathcal N_p}&=&-\partial_u
\end{eqnarray}
for a non-vanishing function $E$ defined on $\mathcal N_p$,
\item [(b)] The fuction $h^{3}_{22}$ satisfies
\begin{equation}
\label{SubsectTisLLLemma2Eq1c} h^{3}_{22}=-\frac{E_{u}}{E}+\dfrac{f^{\prime}}{f}.
\end{equation}
\end{itemize}
\end{lemma}

\begin{proof}
Because of \eqref{SubsectTisLLLCConn2}, we have $[ET,-U]=\left(E\left(\dfrac{f^{\prime}}{f}-h^{3}_{22}\right)+U(E)\right)T$ whenever $E\in C^\infty(M)$. So, if $E$ is a non-vanishing function satisfying
\begin{equation}
\label{SubsectTisLLLemma2PEq1} E\left(\dfrac{f^{\prime}}{f}-h^{3}_{22}\right)+U(E)=0,
\end{equation}
then we have $[ET,-U]=0$. Therefore there exists a local coordinate system  $\left(\mathcal N_p, \left(u,v\right)\right)$ near to $p$ such that $ET=\partial_v$ and $-U=\partial_u$ from which we obtain \eqref{SubsectTisLLLemma2Eq1a} and \eqref{SubsectTisLLLemma2Eq1b}. Consequently, \eqref{SubsectTisLLLemma2Eq1b} and \eqref{SubsectTisLLLemma2PEq1} implies \eqref{SubsectTisLLLemma2Eq1c}.
\end{proof}
	

Next, we have the following local classification theorem for the surfaces satisfying the condition that  $\left(\frac{\partial}{\partial z}\right)^T$ is light-like.
\begin{theorem}\label{SubsectTisLLThm1}
Let $M$ be a time-like surface in  $\mathbb L^3_1(c)\,{}_f\times I$. Then, $\left(\frac{\partial}{\partial z}\right)^T$ is light-like on $M$ if and only if every point $p\in M$ has a neighborhood which can be parametrized as  given in \Cref{SubsectTisLLProp1}.
\end{theorem}

\begin{proof}
	In order to prove the necessary condition, consider a local coordinate system  $\left(\mathcal N_p, \left(u,v\right)\right)$  given in \Cref{SubsectTisLLLemma2} near to an arbitrary $p\in M$. Let $\phi(u,v)=(\tilde{\phi}(u,v),Z(u,v))$ be the parametrization of $\mathcal N_p$, where we define
$$\tilde{\phi}:=\Pi\circ\phi,\qquad Z:=\phi_4$$
and $\Pi: \mathbb L^3_1(c)\times I\to\mathbb L^3_1(c)$ is the canonical projection. Then, by combining \eqref{SubsectTisLLLemma2Eq1a} and  \eqref{SubsectTisLLLemma2Eq1b} with \eqref{SubsectTisLL-Eq1}, we get
\begin{eqnarray*}	
\frac{1}{E}\left(\tilde{\phi}_v,Z_v\right)&=T=&\bar{T},\\
-\left(\tilde{\phi}_u,Z_u\right)&=U=&(\bar{U},-1)
\end{eqnarray*}	
on $\mathcal N_p$ from which we have
\begin{eqnarray}	
\label{SubsectTisLLProp1PrEq1} Z_v=1, &\qquad& Z_u=0,\\
\label{SubsectTisLLProp1PrEq2} \tilde{\phi}_v=\bar{T},&\qquad& -\tilde{\phi}_u=\bar{U}.
\end{eqnarray}	
Because of \eqref{SubsectTisLLProp1PrEq1}, by a translation on the coordinate $u$, we have  $Z(u,v)=u$. Therefore, $\mathcal N_p$ can be parametrized as given in \eqref{SubsectTisLLProp1Eq1}  for an $\mathbb L^3_{1}(c)$-valued smooth function $\tilde{\phi}$. Furthermore, \eqref{SubsectTisLL-Eq1} and \eqref{SubsectTisLLProp1PrEq2} imply
$$g_{c}\left(\tilde{\phi}_u,\tilde{\phi}_u\right)=-\frac{1}{f(u)^2},\quad g_{c}\left(\tilde{\phi}_u,\tilde{\phi}_v\right)=\frac{E}{f(u)^2},\quad g_{c}\left(\tilde{\phi}_v,\tilde{\phi}_v\right)=0.$$
Consequently, $\tilde{\phi}$ is a  parametrization of a Lorentzian surface $\hat M$ in $\mathbb L^3_{1}(c)$ with the induced  metric given by \eqref{SubsectTisLLProp1Eq2}. Thus, the proof of the necessary condition is completed. The proof  of the sufficient condition is obtained in \Cref{SubsectTisLLProp1}.
\end{proof}



\section{Class $\mathcal A$ Surfaces}
In this subsection, we are going to consider class~$\mathcal A$  surfaces in $\mathbb L^3_1(c)\,{}_f\times I$  such that $\left(\frac{\partial}{\partial z}\right)^T$ is light-like.

Let $\hat M$ be a Lorentzian surface in $\mathbb L^{3}_{1}(c)$ with the local parametrization $\tilde\phi(u,v)$ such that \eqref{SubsectTisLLProp1Eq2} and  \eqref{SubsectTisLLProp1Eq1}. Let  $\tilde N$ and $\hat S$ denote the unit normal vector field and shape operator of $\hat M$, respectively. We are going to use the following lemma obtained by a direct computation.
\begin{lemma}\label{SectClassALem1}
		The Levi-Civita connection $\nabla^{\mathbb L^3_1(c)}$ satisfies 
\begin{align}\label{SectClassALem1Eq1}
\begin{split}
\nabla^{\mathbb L^3_1(c)}_{\partial_u}{\partial_u} &= \left(\frac{E_u}{E}-\frac{2f^{\prime}}{f}\right)\tilde{\phi}_{u}+\left(\frac{E_{u}}{E^2}-\frac{f^{\prime}}{fE}\right)\tilde{\phi}_v+h_{1}\tilde{N} \\
\nabla^{\mathbb L^3_1(c)}_{\partial_u}{\partial_v} &= h_{2}\tilde{N}\\
\nabla^{\mathbb L^3_1(c)}_{\partial_v}{\partial_v} &= \frac{E_v}{E}\tilde{\phi}_{v}+h_{3}\tilde{N}
\end{split}
\end{align}
for some smooth functions $h_i$. 
\end{lemma}


Now, consider the surface parametrized by \eqref{SubsectTisLLProp1Eq1}. Then, by \Cref{SubsectTisLLThm1},   $\left(\frac{\partial}{\partial z}\right)^T$ is light-like on $M$. On the other hand, the vector fields
\begin{align*}
\begin{split}
T=\frac{1}{E}\partial_v=\frac{1}{E}(\tilde\phi_v,0),\quad & U=-\partial_u=-(\tilde\phi_u,1),\\
e_3=\frac{1}{E}(\tilde\phi_v,1),\quad & e_4=\frac 1f(\tilde N,0)\\
\end{split}
\end{align*}
form the frame field $\{T,U,e_3,e_4\}$ defined by \eqref{partialzDecompT-LL}. By a direct computation, we obtain
		\begin{equation} \label{SectClassAshapeoperators_new}
			A_{e_3} = \begin{pmatrix}
				-\frac{f^{\prime}}{f} & \frac{E_u}{E}-\frac{f^{\prime}}{f}\\
				0 & -\frac{f^{\prime}}{f}
			\end{pmatrix}, \quad
			A_{e_4} = \begin{pmatrix}
				\frac{fh_2}{E} & -fh_1 \\
				-\frac{fh_3}{E^{2}} & \frac{fh_2}{E}
			\end{pmatrix}.
		\end{equation}
		
Next, we get the following corollary by considering the shape operators of $M$ given by \eqref{SectClassAshapeoperators_new}.
\begin{corry}\label{SectClassACorr1}
Let $M$  be a surface in $\mathbb L^3_1(c)\,{}_f\times I$  given in \Cref{SubsectTisLLProp1} with light-like  $\left(\frac{\partial}{\partial z}\right)^T$. Then, the followings are equivalent to each other:
\begin{enumerate}
\item[(i)] $M$ is a class $\mathcal{A}$ surface, 
\item[(ii)] $h_3=0$,
\item[(iii)] $\partial_v$ is a principal direction of $\hat M$.
\end{enumerate}
\end{corry}
\begin{proof}
\textbf{$(i)\Leftrightarrow (ii)$}	Because of the first equation of \eqref{SectClassAshapeoperators_new}, we have $A_{e_3}T=-\frac{f^{\prime}}{f}$, that is $T$ is a principal direction of $A_{e_3}$. Therefore, \eqref{SectClassALem1Eq1} and the second equation of \eqref{SectClassAshapeoperators_new} imply that 	$M$  is a $\mathcal{A}$ surface if and only if $h_3=0$.

\textbf{$(ii)\Leftrightarrow (iii)$} By the definition of $h_3$, we have $h_3=0$ is satisfied if and only if
$$g_{c}\left(\nabla^{\mathbb L^3_1(c)}_{\partial_v}{\partial_v},\tilde{N}\right)=0.$$
which is equivalent to 
$$\hat S(\partial_v)=\frac{Eh_2}{f^2}\partial_v.$$
\end{proof}

Now, we are ready to prove the main result of this paper.
\begin{theorem}\label{L31cfI-ClassA-MainThm}
Let $M$ be a surface in the space-time $\mathbb L^3_1(c)\, {}_f\times I$ such that $\left(\frac{\partial}{\partial z}\right)^T$ is light-like. Then, $M$ is a class~$\mathcal A$  surface if and only if for all $p\in M$ there exists a neighborhood $\mathcal N_p$ given by one of following five cases:
\begin{enumerate}
\item[(i)] $\mathcal N_p$ is congruent to the surface given in  \Cref{SubsectTisLLProp1L31c},
\item[(ii)] $c=1$ and $\mathcal N_p$ is congruent to the surface given in  \Cref{SubsectTisLLProp1S31},
\item[(iii)] $c=-1$ and $\mathcal N_p$ is congruent to the surface given in  \Cref{SubsectTisLLProp1H31},
\item[(iv)] $c=-1$ and $\mathcal N_p$ is congruent to the surface given in  \Cref{SubsectTisLLProp1H31Quad},
\item[(v)] $c=0$ and $\mathcal N_p$ is congruent to the surface given in  \Cref{SubsectTisLLProp1E31}.
\end{enumerate}
\end{theorem}
\begin{proof}
Let $p\in M$. Then, since $M$ has light-like $\left(\frac{\partial}{\partial z}\right)^T$, \Cref{SubsectTisLLThm1} implies the existence of a local coordinate system  $\left(\mathcal N_p, \left(u,v\right)\right)$ near to $p$ such that $\mathcal N_p$ is parametrized by \eqref{SubsectTisLLProp1Eq1} for a Lorentzian surface $\hat M$  with the metric tensor \eqref{SubsectTisLLProp1Eq2}. Now, in order to prove the necessary condition, assume that $\mathcal N_p$ is class $\mathcal A$. Then \Cref{SectClassACorr1} implies that the shape operator $\hat S$ of $\hat M$ has a light-like eigenvector $\partial_v$. Therefore, $\hat S$ has the canonical forms  given in case I or III of \eqref{SubSectPrelL31c}.

\noindent\textit{ Case 1}. $\hat S$ is diagonalizable. In this case,  there exists an orthonormal frame field of the tangent bundle of $\hat M$ such that $\hat S$ has the form given in case I  of \eqref{SubSectPrelL31c} for some $\lambda_1,\lambda_2$. Since $\partial_v$ is a light-like eigenvector, we have $\lambda_1=\lambda_2$, i.e, $\hat M$ is (totally) umbilical. By \Cref{L31cRemarkTotallUmb}, we have four subcases.

\noindent\textit{ Case 1.(a)}. $c=1$ and $\hat M\subset \mathbb S^2_1(\frac 1{r^2})\subset \mathbb S^3_1$ for a constant $0<r<1$. In this case, $\mathcal N_p$ can be parametrized by
\begin{align}\label{L31cfI-ClassA-MainThmProCase1aEq1}
\begin{split}
\phi(u,v)=&\Big(\frac{r}{s_1(u,v)},r\left(\frac{ \cos s_2(u,v)}{s_1(u,v)}- \sin s_2(u,v)\right),r \left(\frac{\sin s_2(u,v)}{s_1(u,v)}+\cos s_2(u,v)\right),\\&
\sqrt{1-r^2},u\Big)
\end{split}
\end{align}
for some smooth functions $s_1,s_2$. Since $\partial_u$ and $\partial_v$ are light-like vectors, \eqref{L31cfI-ClassA-MainThmProCase1aEq1} implies
\begin{eqnarray}
\label{L31cfI-ClassA-MainThmProCase1aEq2} \frac{\partial s_2}{\partial v}\left(2\frac{\partial s_1}{\partial v} +\left(s_1{}^2+1\right)\frac{\partial s_2}{\partial v} \right)&=&0,\\
\label{L31cfI-ClassA-MainThmProCase1aEq3} s_1 {}^2+r^2 f ^2 \left(s_1 {}^2+1\right)  \left(\frac{\partial s_2}{\partial u}\right)^2+2 r^2 f ^2  \frac{\partial s_1}{\partial u}   \frac{\partial s_2}{\partial u}&=&0.
\end{eqnarray}
Because of \eqref{L31cfI-ClassA-MainThmProCase1aEq2}, we have either 
\begin{equation}
\label{L31cfI-ClassA-MainThmProCase1aEq4} s_2(u,v)=a(u)
\end{equation}
or
\begin{equation}
\label{L31cfI-ClassA-MainThmProCase1aEq5} s_2(u,v)=a(u)-2 \tan ^{-1}\left(s_1(u,v)\right)
\end{equation}
for a smooth function $a$. 

Let $s_2$ satisfy \eqref{L31cfI-ClassA-MainThmProCase1aEq4}. In this case,  \eqref{L31cfI-ClassA-MainThmProCase1aEq3} implies 
\begin{equation}
\label{L31cfI-ClassA-MainThmProCase1aEq4-1} s_1(u,v)=A(u,v)
\end{equation}
for a function $A$ satisfying \eqref{L31cTotUmbParamAsatisfies}. By combining \eqref{L31cfI-ClassA-MainThmProCase1aEq4}  and \eqref{L31cfI-ClassA-MainThmProCase1aEq4-1} with \eqref{L31cfI-ClassA-MainThmProCase1aEq1}, we obtain \eqref{S31TotUmbParam2}. Therefore, we have the case (ii) of the theorem.

Note that the other case \eqref{L31cfI-ClassA-MainThmProCase1aEq5} results in a surface congruent to the one given in \Cref{SubsectTisLLProp1S31}.


\noindent\textit{ Case 1.(b)}. $c=-1$ and $\hat M\subset \mathbb H^2_1(-\frac 1{r^2})\subset \mathbb H^3_1$  for a constant $r>1$.   In this case, similar to case 1.(b), we start with the parametrization of  $\mathcal N_p$ given by
\begin{align*}
\begin{split}
\phi(u,v)=&\Big(r \left(\frac{\cos s_2(u,v)}{s_1(u,v)}-\sin s_2(u,v)\right),r \left(\frac{\sin s_2(u,v)}{s_1(u,v)}+\cos s_2(u,v)\right),\frac{r}{s_1(u,v)},\\&\sqrt{r^2-1},u\Big)
\end{split}
\end{align*}
for some smooth functions $s_1,s_2$ and obtain \eqref{L31cfI-ClassA-MainThmProCase1aEq2} together with
\begin{eqnarray*}
s_1 {}^2-r^2 f ^2 \left(s_1 {}^2+1\right)  \left(\frac{\partial s_2}{\partial u}\right)^2-2 r^2 f ^2  \frac{\partial s_1}{\partial u}   \frac{\partial s_2}{\partial u}&=&0.
\end{eqnarray*}
Then, we get the surface parametrized by \eqref{H31TotUmbParam2}. Thus, we have the case (iii) of the theorem.


\noindent\textit{ Case 1.(c)}. $c=-1$ and $\hat M\subset \mathbb H^3_1$  is  the flat surface parametrized by \eqref{H31TotUmbQuadraParam1}. Note that the parameters $U,V$ in \eqref{H31TotUmbQuadraParam1} satisfy \eqref{L31c-isothermal-null}. Therefore, because of \Cref{L31fxIFlatNBPropRem}, the coordinates $u,v$ satisfy \eqref{L31fxIFlatNBPropRemEq1}. By combining  \eqref{H31TotUmbQuadraParam1}, \eqref{L31fxIFlatNBPropRemEq1} with \eqref{SubsectTisLLProp1Eq1}, we obtain the surface constructed in \Cref{SubsectTisLLProp1H31Quad}. Therefore, we have the case (iv) of the theorem.

\noindent\textit{ Case 1.(d)}. $c=0$ and $\hat M\subset \mathbb S^2_1(\frac 1{r^2})\subset \mathbb E^3_1$. In this case, by a similar way to case 1.(b), we get the case (v) of the theorem.


\noindent\textit{ Case 2}. $\hat S$ is non-diagonalizable. In this case,  there exists a pseudo-orthonormal frame field of the tangent bundle of $\hat M$ such that $\hat S$ has the form given in case III  of \eqref{SubSectPrelL31c}. In this case, $\hat M$ is a null scroll  described in \Cref{L31cNullScroll} for some functions $a$ and $b$ and it can be parametrized as given in \eqref{L31c-NullScrollParam}.

By considering \eqref{L31c-NullScrollParam} and  \eqref{SubsectTisLLProp1Eq1}, we obtain \eqref{L31cTotUmbParam2} for some smooth functions $U(u,v)$ and $V(u,v)$. Since $\partial_u$ and $\partial_v$ are light-like vectors, \eqref{L31cfI-ClassA-MainThmProCase1aEq1} implies
\begin{eqnarray}
\label{L31cfI-ClassA-MainThmProCase2Eq2} \frac{\partial U}{\partial v} \left( V^2 \left(b(U)^2+c\right)\frac{\partial U}{\partial v}-2 \frac{\partial V}{\partial v}\right)&=&0,\\
\label{L31cfI-ClassA-MainThmProCase3Eq3} f^2 \frac{\partial U}{\partial u} \left( V^2 \left(b(U)^2+c\right)\frac{\partial U}{\partial u}-2 \frac{\partial V}{\partial u}\right)+1&=&0.
\end{eqnarray} 
 Because $\partial_V$ is the only eigenvector of the shape operator of $\hat M$ (See \eqref{L31c-NullScrollParamShapOp}), the vector field $\partial_v$ must be proportional to $\partial_V$. Therefore, \eqref{L31cfI-ClassA-MainThmProCase2Eq2} implies 
\begin{eqnarray} \label{L31cfI-ClassA-MainThmProCase2Eq4}
U=U(u).
\end{eqnarray}
Consequently, \eqref{L31cfI-ClassA-MainThmProCase3Eq3} turns into \eqref{L31cParamAsatisfies}. By combining \eqref{L31cParamAsatisfies} and \eqref{L31cfI-ClassA-MainThmProCase2Eq4} with \eqref{L31cTotUmbParam2}, we get the surface constructed in \Cref{SubsectTisLLProp1L31c}. Therefore, we have case (i) of the theorem. Hence the proof of necessary condition is completed. The proof of the sufficient condition is obtained in the previous subsection.
\end{proof}


\section{Applications of  Class $\mathcal A$ Surfaces}
In this section, as applications of \Cref{L31cfI-ClassA-MainThm} we obtain local clasification theorems on some important classes of surfaces in the space-time $\mathbb L^3_1(c)\, {}_f\times I$ with the property of of having light-like $\left(\frac{\partial}{\partial z}\right)^T$.

Let $M$  be the surface in $\mathbb L^3_1(c)\,{}_f\times I$ given by \Cref{SubsectTisLLProp1}. First, we obtain the following corollaries directly follows from \eqref{SectClassAshapeoperators_new}.
	\begin{corry}\label{SectClassACorr2}
		$M$  is a pseudo-umbilical surface of  $\mathbb L^3_1(c)\,{}_f\times I$ if and only if  the equations
		\begin{equation} \label{SectClassACorr2Eq1}
		 h_2 h_3=0, \qquad h_1 h_2 f^4-E_u f' f+E f'{}^2 =0
		\end{equation}
are satisfied.		
\end{corry}
\begin{proof}
By a direct computation using  \eqref{SectClassAshapeoperators_new}, we obtain
$$A_H=
\begin{pmatrix}
 \frac{f^2 h_2^2}{E^2}+\frac{f'{}^2}{f^2} & -\frac{h_1 h_2 f^4-E_u f' f+E f'{}^2}{E f^2} \\
 -\frac{f^2 h_2 h_3}{E^3} & \frac{f^2 h_2^2}{E^2}+\frac{f'{}^2}{f^2} \\
\end{pmatrix}
$$
which yields that $M$  is  pseudo-umbilical if and only if \eqref{SectClassACorr2Eq1} is satisfied.
\end{proof}

	\begin{corry}\label{SectClassACorr3}
		$M$  is a totally umbilical surface of  $\mathbb L^3_1(c)\,{}_f\times I$ if and only  the functions $E$, $h_1$ and $h_3$ satisfy
		\begin{equation} \label{SectClassACorr3Eq1}
		 \frac{f^{\prime}}{f}-\frac{E_u}{E}=h_1=h_3=0.
		\end{equation}
\end{corry}
\begin{proof}
The proof directly follows from \eqref{SectClassAshapeoperators_new}.
\end{proof}

	\begin{corry}\label{SectClassACorr4}
		$M$  has flat normal bundle in  $\mathbb L^3_1(c)\,{}_f\times I$ if and only  the functions $E$ and $h_3$ satisfy
		\begin{equation} \label{SectClassACorr3Eq4}
		 h_3\left(\frac{f^{\prime}}{f}-\frac{E_u}{E}\right)=0
		\end{equation}
\end{corry}
\begin{proof}
The proof directly follows from \eqref{SectClassAshapeoperators_new} and the Ricci equation \eqref{Ricci}.
\end{proof}


\subsection{Surfaces with Flat Normal Bundle}
In this subsection, as a result of \Cref{SectClassACorr4}, we obtain the following corollary.
\begin{proposition}
Let $M$ be a Lorentzian surface in the static space-time $\mathbb L^3_1(c)\,{}_f\times I$ with light-like $\left(\frac{\partial}{\partial z}\right)^T$. Then $M$ has flat normal bundle if and only if it is locally congruent to one of the following class of surfaces:
\begin{enumerate}
\item[(i)] A class~$\mathcal A$  surface given in \Cref{L31cfI-ClassA-MainThm},
\item[(ii)] A surface which can be parametrized by \eqref{SubsectTisLLProp1Eq1} given in \Cref{SubsectTisLLProp1}, where $\tilde{\phi}(u,v)$ is a local parametrization of a flat Lorentzian surface $\hat M$ in $\mathbb L^{3}_{1}(c)$ with the induced metric given by \eqref{L31fxIFlatNBPropCaseEqIIMetric}.
\end{enumerate}
\end{proposition}
\begin{proof}
Let $p\in M$. Since $M$ has light-like $\left(\frac{\partial}{\partial z}\right)^T$, \Cref{SubsectTisLLThm1} implies that $p$ has a neighborhood $\mathcal N_p$ parametrized by \eqref{SubsectTisLLProp1Eq1} given in \Cref{SubsectTisLLProp1}. 

Now, in order to prove the necessary condition, we assume that $M$ has flat normal bundle. Then, because of \Cref{SectClassACorr4},  we have two cases: If $h_3=0$ on $\mathcal N_p$, then \Cref{SectClassACorr1} implies that $\mathcal N_p$ is class~$\mathcal A$  and we have the case (i). On the other hand, if $h_3(q)\neq0$ at $q\in \mathcal N_p$, then, by shrinking  $\mathcal N_p$, if necessary, we observe that \eqref{SectClassACorr3Eq4} implies 
$$\frac{f^{\prime}}{f}-\frac{E_u}{E}=0$$
on $\mathcal N_p$. So, there exists a non-vanishing function $c_1$ such that $E(u,v)=c_1(v)f(u)$. By re-defining $v$ properly we can assume $c_1=1$. Thus, we have \eqref{L31fxIFlatNBPropCaseEqIIMetric}. Consequently,  $\mathcal N_p$ is flat and we have case (ii) of the theorem.

Converse follows from a direct computation.
\end{proof}

Note that \Cref{L31fxIFlatNBPropRem} ensures the existence of a local coordinate system for which \eqref{L31fxIFlatNBPropCaseEqIIMetric} is satisfied.

\subsection{Pseudo-Umbilical Surfaces}
In this subsection, we consider the  pseudo-umbilical surfaces with light-like $\left(\frac{\partial}{\partial z}\right)^T$ in $\mathbb L^3_1(c)\,{}_f\times I$.

We are going to use the following lemma.
\begin{lemma}\label{L31cPseudoUmbLemma1}
Let $M$ be a class~$\mathcal A$  surface in $\mathbb L^3_1(c)\,{}_f\times I$ with  light-like $\left(\frac{\partial}{\partial z}\right)^T$. Then, $M$ is pseudo-umbilical if and only if it is congruent to the surface given in  \Cref{SubsectTisLLProp1L31c} for some $a,b,U,V$ satisfying \eqref{L31cParamAsatisfies} and
\begin{subequations}\label{L31cPseudoUmbLemma1Eq1ALL}
\begin{eqnarray}
\label{L31cPseudoUmbLemma1Eq1a} U' b'(U)b(U)-\frac{f' \left(b(U)^2+c\right)}{f}&=&0,\\
\label{L31cPseudoUmbLemma1Eq1b} a(U) b(U) U'{}^2+\frac{b(U)^2}{f^2}-\frac{f' \left(f' U'+f U''\right)}{f^2 U'}&=&0.
\end{eqnarray}
\end{subequations}
\end{lemma}

\begin{proof}
If  $M$ is pseudo-umbilical. Then it is locally congruent to one of five surfaces given in \Cref{L31cfI-ClassA-MainThm}. First, we are going to prove that the surface given in case (ii)-(v) can not be pseudo-umbilical. 

Towards contradiction, assume that $M$ is the surface given in \Cref{SubsectTisLLProp1S31} and it is pseudo-umbilical. Then, \eqref{SectClassACorr2Eq1} is satisfied because of \Cref{SectClassACorr2}. By a direct computation using \eqref{L31cTotUmbShape}, \eqref{L31cTotUmbShape2} and \eqref{SectClassACorr2Eq1}, we obtain
$$\frac{f(u) a'(u) f'(u)}{U(u,v)}+\frac{f(u) a''(u) f'(u)}{a'(u)}+f'(u)^2-\frac{1}{r^2}+1=0$$
which can be satisfied only if $a'(u) f'(u) U_v(u,v)=0$. However, this is not possible because of \eqref{L31cTotUmbShape0}. By a similar method, we observe that none of the surfaces constructed in \Cref{SubsectTisLLProp1H31}, \Cref{SubsectTisLLProp1H31Quad}, and \Cref{SubsectTisLLProp1E31} is pseudo-umbilical
because of \eqref{L31cTotUmbShape} and \eqref{H31TotUmbQuadShpOp}.

Consequently, $M$ is locally congruent to a surface given in  \Cref{SubsectTisLLProp1L31c} for some smooth functions $a,b,U,V,E$ satisfying  \eqref{L31cParamAsatisfies} and \eqref{L31cParamEDef}. Moreover, from \eqref{L31cB-ScrollShape} and \eqref{SectClassACorr2Eq1} we see that $M$ is pseudo-umbilical if and only if the equation
\begin{equation}\label{L31cPseudoUmbLemma1PEq1} 
VU' \left(b(U) U' b'(U)-\frac{f' \left(b(U)^2+c\right)}{f}\right)+a(U) b(U) \left(U'\right)^2+\frac{b(U)^2}{f^2}-\frac{f' \left(f' U'+f U''\right)}{f^2 U'}=0
\end{equation}
is satisfied. Note that \eqref{L31cParamEDef} implies that $V_v\neq0$. Therefore, \eqref{L31cPseudoUmbLemma1PEq1} implies  \eqref{L31cPseudoUmbLemma1Eq1ALL}.
\end{proof}


Next, we obtain the local classification of pseudo-umbilical surfaces in $\mathbb{E}^{3}_1(c)\,{}_f\times I$.
\begin{proposition}\label{E31fxIPseuUmbProp}
Let $M$ be a surface in the static space-time $\mathbb{E}^{3}_1(c)\,{}_f\times I$ with light-like $\left(\frac{\partial}{\partial z}\right)^T$. Then $M$ is pseudo-umbilical if and only if it is locally congruent to one of the following surfaces:
\begin{enumerate}
\item[(i)] $M$ is a class~$\mathcal A$  surface parametrized by 
\begin{align}\label{E31fxIPseuUmbPropCaseIParam} 
\begin{split}
\phi(u,v)=& \left(\frac{(c_1 F(u)+c_2)^3+6 (c_1 F(u)+c_2)+\frac{3 F(u)-6 v}{c_1}}{6 \sqrt{2}},\frac{1}{2} (c_1 F(u)+c_2)^2,\right.\\&\left.
\frac{(c_1 F(u)+c_2)^3-6 (c_1 F(u)+c_2)+\frac{3 F(u)-6 v}{c_1}}{6 \sqrt{2}},u\right),
\end{split}
\end{align}
\item[(ii)] $M$ is a class~$\mathcal A$  surface given in \Cref{SubsectTisLLProp1L31c} for some $a,b,U,V$ satisfying
\begin{eqnarray}
\label{E31fxIPseuUmbPropCaseIICond1} b(U)&=&c_3 f, \qquad V_u=\frac{c_3^2}{2}  f^2 U V^2+\frac{1}{2 f^2 U}\\
\label{E31fxIPseuUmbPropCaseIICond2} a(U)&=&\frac{-c_3^2 f^2 U'+f f' U''+f'^2 U'}{c_3 f^3 U'^3}
\end{eqnarray}
for a non-zero constant $c_3$.
\item[(iii)] A flat surface generated by a suitable cone in $\mathbb E^3_1$ which can be parametrized by 
\end{enumerate}
\begin{eqnarray}
\label{E31fxIPseuUmbPropCaseIII} \phi(u,v)&=&\left(F(u)-v,b_1(v),b_2(v),u\right)
\end{eqnarray}
for some smooth functions $b_1,\ b_2$ such that $b_1'{}^2+b_2'{}^2=1$.
\end{proposition}

\begin{proof}
Let $p \in M$. Since the  vector field $\left(\frac{\partial}{\partial z}\right)^T$ is light-like on $M$, \Cref{SubsectTisLLThm1} ensures that a neighborhood $\mathcal N_p$ of $p$ which can be parametrized by \eqref{SubsectTisLLProp1Eq1} presented in \Cref{SubsectTisLLProp1} for a Lorentzian surface  $\hat M$. In order to prove the necessary condition, we assume that  $\mathcal N_p$ is pseudo-umbilical. Then, \Cref{SectClassACorr2} implies that we have two cases.

\textit{Case I.} $\mathcal N_p$ is class $A$. In this case, \Cref{L31cPseudoUmbLemma1} implies that  is congruent to the surface given in 
\Cref{SubsectTisLLProp1L31c}, for some $a,b,U,V$ satisfying \eqref{L31cParamAsatisfies} and \eqref{L31cPseudoUmbLemma1Eq1ALL} and  $\hat M$ is a null scroll. Since $c=0$, \eqref{L31cPseudoUmbLemma1Eq1a} implies the first equation in \eqref{E31fxIPseuUmbPropCaseIICond1} for a constant $c_3$. We have two sub-cases:

\textit{Case I.a.} $c_3=0$. In this case, we have $b=0$. Thus, $\hat M$ is flat which yields that $\left(U,V\right)$ satisfies for \eqref{L31fxIFlatNBPropRemEq1} for some constants $c_1\neq0,c_2$.

When $a\neq0$, by \Cref{FlatMinimalB-scrolls}, $\hat M$ is congruent to the B-scroll parametrized by \eqref{E31-FlatBScrollParam}.  By combining \eqref{E31-FlatBScrollParam} with \eqref{L31fxIFlatNBPropRemEq1} we get \eqref{E31fxIPseuUmbPropCaseIParam}. So, we have case (i) of the theorem.

On the other hand, if $a=0$, then \eqref{L31cB-ScrollShape} implies that $A_{e_4}=0$ and (4) of \Cref{SubsectTisLLLemma1} implies $\nabla^\perp e_4=0$.  Therefore,  $\mathcal N_p$ lies on the totally geodesic hypersurface $\mathcal M\times I$ of $\mathbb{E}^{3}_1(c)\,{}_f\times I$, where $\mathcal M$ is the plane in $\mathbb E^3_1$ with the normal $\bar{e_4}$. However, this is a contradiction.

\textit{Case I.b.} $c_3\neq 0$. In this case, \eqref{L31cParamAsatisfies} and \eqref{L31cPseudoUmbLemma1Eq1b} gives the second equation in \eqref{E31fxIPseuUmbPropCaseIICond1} and \eqref{E31fxIPseuUmbPropCaseIICond2}, respectively. So, we have case (ii) of the theorem.

\textit{Case II.} $\mathcal N_p$ is not class $A$, i.e., $h_3\neq0$. In this case, \eqref{SectClassACorr2Eq1} implies 
\begin{eqnarray}\label{E31fxIPseuUmbPropCaseIIIProofEq0}
h_2=\frac{f^{\prime}}{f}-\frac{E_u}{E}&=&0.
\end{eqnarray}
Therefore, $\hat M$ is flat, we have $E=f$ and \eqref{L31fxIFlatNBPropCaseEqIIMetric}. Moreover, from \Cref{SectClassALem1}  we have 
\begin{eqnarray}
\label{E31fxIPseuUmbPropCaseIIIProofEq1a}\tilde \phi_{uu} &=& -\frac{f^{\prime}}{f}\tilde{\phi}_{u}+h_{1}\tilde{N},\\
\label{E31fxIPseuUmbPropCaseIIIProofEq1b}\tilde \phi_{uv} &=& 0
\end{eqnarray}
and \eqref{L31fxIFlatNBPropCaseEqIIMetric} implies
\begin{equation}
\label{E31fxIPseuUmbPropCaseIIIProofEq1d} N_{u} = -h_{1}f\tilde\phi_v,\qquad N_{v} = -h_{3}(f\tilde\phi_u+\tilde\phi_v).
\end{equation}
Therefore, $\hat M$ is flat and because of \eqref{E31fxIPseuUmbPropCaseIIIProofEq1d},  the Gauss equation \eqref{GaussEq} implies $h_1h_3=0$ which implies $h_1=0$. Moreover, the Codazzi equation \eqref{Codazzi} for $X=Z=\partial_u$, $Y=\partial_v$ gives $h_1=h_1(u)$. Thus, by solving \eqref{E31fxIPseuUmbPropCaseIIIProofEq1a} and \eqref{E31fxIPseuUmbPropCaseIIIProofEq1b}, we obtain
\begin{eqnarray}
\label{E31fxIPseuUmbPropCaseIIIProofEq2}
\tilde\phi(u,v)&=&F(u)v_1+\phi_2(v)
\end{eqnarray}
for a constant vector $v_1$ and a $\mathbb E^3_1$-valued smooth function $\phi_2$. \eqref{L31fxIFlatNBPropCaseEqIIMetric} and \eqref{E31fxIPseuUmbPropCaseIIIProofEq2} imply 
$$g_0(v_1,v_1)=-1,\ g_0(v_1,\phi_2')=1,\ g_0(\phi_2',\phi_2')=0.$$
 Therefore up to a suitable isometry of $\mathbb E^3_1$, one can choose 
\begin{eqnarray}
\label{E31fxIPseuUmbPropCaseIIIProofEq3}
v_1=(1,0,0),\qquad \phi_2(v)=(-v,b_1(v),b_2(v)),\ b_1'{}^2+b_2'{}^2=1.
\end{eqnarray}
By combining \eqref{E31fxIPseuUmbPropCaseIIIProofEq2} and \eqref{E31fxIPseuUmbPropCaseIIIProofEq3} we obtain \eqref{E31fxIPseuUmbPropCaseIII} which gives case (iii) of the proposition. Hence, the proof of necessary condition completed. 

Conversely, a direct computation yields that if $M$ is a surface given by case (i) or case (iii) of the theorem, then the shape operator $A_H$ of $M$ along $H$ takes the form
$$A_H=\left(\frac{f'}f\right)^2I$$
and for the surface $M$ we obtain
$$A_H=\left(c_3^2+\frac{f'(u)^2}{f(u)^2}\right)I.$$
Hence, all of the surfaces given in the proposition are pseudo-umbilical.
\end{proof}


In the following proposition, we consider pseudo-umbilical surfaces in $\mathbb{S}^{3}_1(c)\,{}_f\times I$.
\begin{proposition}
Let $M$ be a surface in the static space-time $\mathbb{S}^{3}_1(c)\,{}_f\times I$ with light-like $\left(\frac{\partial}{\partial z}\right)^T$. Then $M$ is pseudo-umbilical if and only if it is locally congruent to one of the following surfaces:
\begin{enumerate}
\item[(i)] $M$ is a class~$\mathcal A$  surface given in \Cref{SubsectTisLLProp1L31c} for some $a,b,U,V$ satisfying
\begin{eqnarray}
\label{S31fxIPseuUmbPropCaseIICond1} b(U)&=&\sqrt{c_3 f^2-1}, \qquad V_u=\frac{c_3 f^2}{2}  V^2 U'+\frac{1}{2 f^2 U'}\\
\label{S31fxIPseuUmbPropCaseIICond2} a(U)&=&\frac{U' \left(-c_3 f^2+f'{}^2+1\right)+f f' U''}{f^2 \sqrt{c_3 f^2-1} \left(U'\right)^3}
\end{eqnarray}
for a non-zero constant $c_3$.
\item[(ii)] A flat surface  generated by the flat isoparametric surface $\mathbb S^1_1\left(\frac1{\cos{\theta}}\right)\times \mathbb S^1\left(\frac1{\sin{\theta}}\right)\subset\mathbb S^3_1$ which can be parametrized by 
\begin{align}
\begin{split}
\label{S31fxIPseuUmbPropCaseII} \phi(u,v)=&\left(-\cos \theta  \sinh \left(-\frac{c_2 \sec \theta }{\sqrt{2}}+F(u) (-\cos \theta ) \sqrt{\sec (2 \theta )}+\frac{v \sec \theta }{\sqrt{\sec (2 \theta )}}\right),\right.\\&\left.
\cos \theta  \cosh \left(-\frac{c_2 \sec \theta }{\sqrt{2}}+F(u) (-\cos \theta ) \sqrt{\sec (2 \theta )}+\frac{v \sec \theta }{\sqrt{\sec (2 \theta )}}\right),\right.\\&\left.
\sin \theta  \cos \left(-\frac{c_2 \csc \theta }{\sqrt{2}}+F(u) \sin \theta  \sqrt{\sec (2 \theta )}+\frac{v \csc \theta }{\sqrt{\sec (2 \theta )}}\right),\right.\\&\left.
-\sin \theta  \sin \left(-\frac{c_2 \csc \theta }{\sqrt{2}}+F(u) \sin \theta  \sqrt{\sec (2 \theta )}+\frac{v \csc \theta }{\sqrt{\sec (2 \theta )}}\right),u\right)
\end{split}
\end{align}
for some $\theta\in(0,\pi/4)$ and $c_2\in\mathbb R$.

\end{enumerate}
\end{proposition}

\begin{proof}
Let $p \in M$. Similar to the proof of \Cref{E31fxIPseuUmbProp}, we consider the local parametrization of $M$ on a neighborhood $\mathcal N_p$ of $p$ given in \Cref{SubsectTisLLProp1} for a Lorentzian surface  $\hat M$ and assume that  $\mathcal N_p$ is pseudo-umbilical. We study two cases obtained from \Cref{SectClassACorr2}:

\textit{Case I.} $\mathcal N_p$ is class $A$. Because of \Cref{L31cPseudoUmbLemma1}, $\hat M$ is a null scroll parametrized by \eqref{L31c-NullScrollParam} and  $\mathcal N_p$  is the surface constructed in
\Cref{SubsectTisLLProp1L31c}, for some $a,b,U,V$ satisfying \eqref{L31cParamAsatisfies}. By solving \eqref{L31cPseudoUmbLemma1Eq1a} for $c=1$, we see that $b$ satisfies the first equation in \eqref{S31fxIPseuUmbPropCaseIICond1} for a constant $c_3$ and $\varepsilon=\pm1$. By re-defining $V$ properly on \eqref{L31c-NullScrollParam}, one can assume $\varepsilon=1$. So, we have  the first equation in \eqref{S31fxIPseuUmbPropCaseIICond1}. By a further computation using \eqref{L31cParamAsatisfies} and \eqref{L31cPseudoUmbLemma1Eq1b}, we also get the second equation in \eqref{S31fxIPseuUmbPropCaseIICond1} and \eqref{S31fxIPseuUmbPropCaseIICond2}, respectively. So, we have case (i) of the theorem.

\textit{Case II.} $\mathcal N_p$ is not class $A$, i.e., $h_3\neq0$. In this case, similar to the proof of \Cref{E31fxIPseuUmbProp}, we have \eqref{E31fxIPseuUmbPropCaseIIIProofEq0} from which we observe that  $\hat M$ is flat and the equations \eqref{L31fxIFlatNBPropCaseEqIIMetric}, \eqref{E31fxIPseuUmbPropCaseIIIProofEq1d} are satisfied for some functions $h_1,h_3$. 

By taking into account \eqref{SectClassALem1Eq1} in \Cref{SectClassALem1} and \eqref{E31fxIPseuUmbPropCaseIIIProofEq1d}, we use the Codazzi and Gauss equations \eqref{GaussEq} and \eqref{Codazzi} to obtain $h_1=\frac{1}{k_1f^2},\ h_3=k_1$. Thus, because of  \eqref{E31fxIPseuUmbPropCaseIIIProofEq1d}, the shape operator $\hat S$ of $\hat M$ takes the form
\begin{equation}\label{S31fxIPseuUmbPropEq1}\hat S=\begin{pmatrix}
 0 & k_1 f \\
 -\frac{1}{2 fk_1} & k_1 \\
\end{pmatrix}\qquad\mbox{with respect to $\{\partial_u,\partial_v\}$.}
\end{equation}
for a constant $k_1\neq 0$. Therefore, $\hat M$ is a isoparametric surface with distinct real principal curvatures which yields that $\hat M$ is an open part of $\mathbb S^1_1\left(\frac1{\cos{\theta}}\right)\times \mathbb S^1\left(\frac1{\sin{\theta}}\right)$,  \cite{LiWang2005}. So, up to a suitable isometry of $\mathbb S^3_1$, we can assume that  $\hat M$ can be parametrized by 
\begin{align}\label{S11rxSRParam}
\begin{split}
\hat\psi(U,V)=&\left(\cos \theta  \sinh \frac{U+V}{\sqrt{2}\cos \theta},\cos \theta  \cosh \frac{U+V}{\sqrt{2}\cos \theta},\sin \theta  \cos \frac{U-V}{\sqrt{2}\sin\theta},\right.\\&\left.\sin \theta  \sin \frac{U-V}{\sqrt{2}\sin\theta}\right),
\end{split}
\end{align}
where $\theta\in(0,\pi/2)$ is defined by $k_1=\cot \theta-\tan \theta$.

Note that the shape operator $\hat S$ of $\hat M$ is
\begin{equation}\label{S31fxIPseuUmbPropEq2}
\hat S=\begin{pmatrix}
 \frac{1}{2} (\cot (\theta )-\tan (\theta )) & -\csc (2 \theta ) \\
 -\csc (2 \theta ) & \frac{1}{2} (\cot (\theta )-\tan (\theta )) \\
\end{pmatrix}\qquad\mbox{with respect to $\{\partial_U,\partial_V\}$.}
\end{equation}
 Moreover, by \Cref{L31fxIFlatNBPropRem}, the coordinate systems $(u, v)$ and $\left(U,V\right)$ are related by \eqref{L31fxIFlatNBPropRemEq1} for some $c_1\neq 0,\ c_2 $. By a direct computation considering \eqref{L31fxIFlatNBPropRemEq1}, \eqref{S31fxIPseuUmbPropEq1} and \eqref{S31fxIPseuUmbPropEq2}, we obtain $\theta\in(0,\pi/4)$ and 
\begin{equation}\label{S31fxIPseuUmbPropEq3}
c_1= \sqrt{\frac{\sec (2 \theta )}{2}}.
\end{equation}

By combining  \eqref{S11rxSRParam} with \eqref{L31fxIFlatNBPropRemEq1} and  \eqref{S31fxIPseuUmbPropEq3}, we obtain \eqref{S31fxIPseuUmbPropCaseII} which give the case (ii) of the proposition. Hence, the proof of the necessary condition is completed. The proof of the sufficient condition follows from a direct computation.
\end{proof}

\subsection{Totally-Umbilical Surfaces}
In this subsection, we consider the  totally-umbilical surfaces.

Let $M$ be a  totally-umbilical surface in $\mathbb L^3_1(c)\,{}_f\times I$ with light-like $\left(\frac{\partial}{\partial z}\right)^T$. Then, $M$ is also  pseudo-umbilical. Furthermore, \Cref{SectClassACorr1} and \Cref{SectClassACorr3} implies that it is class~$\mathcal A$. Therefore, $M$ is locally congruent to a surface  given in  \Cref{SubsectTisLLProp1L31c} for some smooth functions $a,b,U,V,E$ satisfying  \eqref{L31cParamAsatisfies} and \eqref{L31cParamEDef}. Then, from \eqref{L31cB-ScrollShape} and \eqref{SectClassACorr3Eq1} imply that have 
\begin{align}\label{TotallyUmbEq1}
\begin{split}
V U' \left(b(U)^2+c\right)+\frac{f'}{f}+\frac{U''}{U'}=&0, \\
f^2 U'{}^2 \left(a(U)+V b'(U)\right)+b(U)=&0.
\end{split}
\end{align}
As described in proof of \Cref{L31cPseudoUmbLemma1}, we have $V_v\neq0$ and $U'\neq0$. Therefore, \eqref{TotallyUmbEq1} gives
\begin{subequations}\label{TotallyUmbEq2ALL}
\begin{eqnarray}
\label{TotallyUmbEq2a} \left(b(U)^2+c\right)=&0,, \\
\label{TotallyUmbEq2b} \frac{f'}{f}+\frac{U''}{U'}=&0, \\
\label{TotallyUmbEq2c} f^2 U'{}^2 a(U)+b(U)=&0.
\end{eqnarray}
and \eqref{L31cParamAsatisfies}, together with \eqref{TotallyUmbEq2a}, implies
\begin{equation}\label{TotallyUmbEq2d}
2f^2U' V_u=1.
\end{equation}
\end{subequations}

We are ready to prove the following proposition:
\begin{proposition}
Let $M$ be a surface in the static space-time $\mathbb L^{3}_1(c)\,{}_f\times I$ with light-like $\left(\frac{\partial}{\partial z}\right)^T$. Then, we have the followings:
\begin{enumerate}
\item[(i)] If $c=0$ or $c=1$, then $M$ can not be totally umbilical.

\item[(ii)] $M$ is a totally umbilical surface of $\mathbb{H}^{3}_1(c)\,{}_f\times I$ if and only if it is congruent to the flat surface given by 
\begin{align}\label{H31fxITotUmbPropEqCase-ii} 
\begin{split}
\phi(u,v)=& \left(\frac{\left(2 k v+k_2\right) \cos \left(F(u)+\frac{k_2}{k}\right)-2 k \sin \left(F(u)+\frac{k_2}{k}\right)}{2 k},\right.\\&\left.
\frac{-k \left(2 k^2+1\right) \cos \left(F(u)+\frac{k_2}{k}\right)-\left(2 k v+k_2\right) \sin \left(F(u)+\frac{k_2}{k}\right)}{2 \sqrt{2} k^2},\right.\\&\left.
\frac{k \left(2 k^2-1\right) \cos \left(F(u)+\frac{k_2}{k}\right)-\left(2 k v+k_2\right) \sin \left(F(u)+\frac{k_2}{k}\right)}{2 \sqrt{2} k^2},\right.\\&\left.
\frac{\left(2 k v+k_2\right) \cos \left(F(u)+\frac{k_2}{k}\right)}{2 k},u\right)
\end{split}
\end{align}

\end{enumerate}
\end{proposition}
\begin{proof}
(i) follows from \eqref{TotallyUmbEq2a} and \eqref{TotallyUmbEq2c}. First, we are going to prove the necessary condition of (ii).

 Assume that $c=-1$ and $M$ is a totally umblical  surface of  $\mathbb{H}^{3}_1(c)\,{}_f\times I$, i.e., \eqref{TotallyUmbEq2ALL} is satisfied. In this case, as described above,  $M$ is locally congruent to a surface  given in  \Cref{SubsectTisLLProp1L31c} for some smooth functions $a,b,U,V,E$ satisfying  \eqref{L31cParamAsatisfies} and \eqref{L31cParamEDef}. Note that the equation \eqref{TotallyUmbEq2a} implies $b=1$ which yields that $\hat M$ is flat.

Next, by combining \eqref{L31fxIFlatNBPropRemEq1} given in \Cref{L31fxIFlatNBPropRem} with \eqref{TotallyUmbEq2b}-\eqref{TotallyUmbEq2d} we get
\begin{align}\label{H31fxITotUmbPropEqProof1} 
\begin{split}
a=-\frac 1k^2,\quad U=k F(u)+k_2, \quad V=\frac{F(u)-2 v}{2 k}
\end{split}
\end{align}
for some constants $k$ and $k_2$. Furthermore, since  $b=1$, the first equation in \eqref{H31fxITotUmbPropEqProof1} and  \Cref{FlatNullScrollH31-Lemma}
imply that  $\hat M$ is congruent to \eqref{FlatNullScrollH31-LemmaEq1}. By combining \eqref{FlatNullScrollH31-LemmaEq1} and \eqref{H31fxITotUmbPropEqProof1}, we get \eqref{H31fxITotUmbPropEqCase-ii}. Hence the proof is completed.

Conversely, let $M$ be a surface given by \eqref{H31fxITotUmbPropEqCase-ii}. By a direct computation, we obtain
$$A_{e_3}=-\frac{f'}{f}I,\qquad A_{e_4}= \frac 1fI$$
which yields that $M$ is totally umbilical.
\end{proof}


\section*{Acknowledgements}
This work forms a part of the first-named author’s PhD thesis and was carried out within the scope of a project supported by T\"UB\.ITAK, the Scientific and Technological Research Council of T\"urkiye (Project number 121F352).

\section*{Declarations}

\textbf{Use of LLMs.} During the preparation of this work, the authors used ChatGPT solely for grammatical checking.

\textbf{Data Availability.} Data sharing not applicable to this article because no datasets were generated or analysed during the current study.

\textbf{Code availability.} N/A.

\textbf{Conflicts of interest.} The authors declare that they have no conflict of interest.

\bibliographystyle{plain}
\bibliography{EJPAM_Sample}

\end{document}